\documentclass{article}
\usepackage[dvipsnames]{xcolor}
\usepackage{caption}
\usepackage{appendix}

\usepackage{comment}
\usepackage{tikz-cd}
\usepackage{subcaption}
\usepackage[percent]{overpic}

\usepackage{fullpage}
\usepackage{amsmath}
\usepackage{amsfonts}
\usepackage{amssymb}
\usepackage{mathtools}
\usepackage{stmaryrd}
\usepackage{bm}
\usepackage{amsthm}
\usepackage{xcolor} 

\usepackage{graphicx}
\usepackage{caption}

\newcommand{\nc}{\newcommand}

\nc{\dmo}{\DeclareMathOperator}
\nc{\nt}{\newtheorem}

\nt{theorem}{Theorem}[section]
\nt{corollary}[theorem]{Corollary}
\nt{observation}[theorem]{Observation}
\nt{lemma}[theorem]{Lemma}
\nt{criterion}[theorem]{Criterion}
\nt*{theorem*}{Theorem}
\nt{question}[theorem]{Question}
\nt{proposition}[theorem]{Proposition} 
\newtheorem*{maintheorem}{Main Theorem}

\newtheoremstyle{named}{}{}{\itshape}{}{\bfseries}{.}{.5em}{\thmnote{#3}}
\theoremstyle{named}

\theoremstyle{definition}
\nt{exercise}{Exercise}
\counterwithin{figure}{section}

\dmo{\Push}{Push}
\dmo{\Mod}{Mod}
\dmo{\Id}{Id}
\dmo{\Aut}{Aut}
\dmo{\Out}{Out}
\dmo{\Teich}{Teich}
\dmo{\Homeo}{Homeo}
\dmo{\BHomeo}{BHomeo}
\dmo{\PMod}{PMod}
\dmo{\SMod}{SMod}
\dmo{\SI}{{\mathcal{SI}}}
\dmo{\I}{\mathcal{I}}
\dmo{\Sp}{Sp}
\dmo{\spl}{\mathfrak{sp}}
\dmo{\PSp}{PSp}
\dmo{\PSL}{PSL}
\dmo{\SL}{SL}
\dmo{\GL}{GL}
\dmo{\PB}{PB}
\dmo{\B}{B}
\dmo{\D}{D}
\dmo{\Brun}{Brun}
\dmo{\sign}{sign}

\nc{\algam}{\alpha \mbox{-}\Gamma}
\nc{\BI}{\mathcal{BI}}
\nc{\Moduli}{\Moody}
\nc{\Moody}{\mathbb M}
\nc{\C}{\mathbb C}
\nc{\Z}{\mathbb Z}
\nc{\Q}{\mathbb Q}
\nc{\R}{\mathbb R}
\nc{\F}{\mathcal F}
\nc{\mcP}{\mathcal{P}}
\nc\Heven{H_1(D_n^0; \Z/2)^{\textrm{even}}}

\nc{\flm}{\lambda_{2}}

\nc{\normalclosure}[1]{\ensuremath{\left \langle \left \langle #1 \right \rangle \right \rangle}}

\nc{\p}[1]{\paragraph{{\bf #1}}}

\newtheorem{innercustomgeneric}{\customgenericname}
\providecommand{\customgenericname}{}
\newcommand{\newcustomtheorem}[2]{%
  \newenvironment{#1}[1]
  {%
   \renewcommand\customgenericname{#2}%
   \renewcommand\theinnercustomgeneric{##1}%
   \innercustomgeneric
  }
  {\endinnercustomgeneric}
}

\newcustomtheorem{customthm}{Theorem}
\newcustomtheorem{reference}{Reference}
\newcustomtheorem{customcor}{Corollary}
\newcustomtheorem{open-problem}{Open Problem}

\usepackage{lineno}
\usepackage{xifthen}

\title{\bf The Burau representation is faithful for $n = 4$}
\author{Vasudha Bharathram \and Joan S. Birman \and Tara E. Brendle}

\begin{document}

\date{\today}

\maketitle

\begin{abstract}
    In this paper we use ideas introduced earlier by Moody \cite{moody}, Long \cite{long}, Long--Paton \cite{Long-Paton} and Bigelow \cite{bigelow} to prove the theorem of the title, that the Burau representation of the classical braid group $\B_4 $ is faithful. An immediate corollary is that the Jones representation of $\B_4$ is also faithful.  
\end{abstract}

\section{Introduction}
\label{section:intro}

Let $D_n$ be a disk with $n$ marked points ${p_1,p_2, \dots, p_n}$ in its interior.  The {\it braid group} $\B_n$  will be interpreted in this paper as the mapping class group $\Mod(D_n)$.  In 1935 Werner Burau introduced a representation of $\B_n$ into $\GL_{n-1}(\Z[t,t^{-1}])$, now known as the reduced Burau representation \cite{Burau}.  We will work mainly with the unreduced Burau representation $\rho_n:\B_n \to \GL_n(\Z[t,t^{-1}])$, and henceforth we refer to $\rho_n$ simply as the Burau representation. 
 
Magnus--Peluso first established the faithfulness of the Burau representation for $n = 3$ by a direct algebraic calculation \cite{magnuspeluso}.   Moody later proved that $\rho_n$ is not faithful for $n \geq 9$ \cite{moody}.  Long--Paton built on Moody's ideas to improve this to $n \geq 6$ \cite{Long-Paton}, and Bigelow later added the case $n = 5$ \cite{bigelow}.  

In the years following Bigelow's paper there have been many attempts to detect elements in the kernel of the remaining case $\rho_4$, or to narrow the search for such,  each of which contributed to our understanding of this question as viewed through the lens of various areas of mathematics; see for example Alperin--Farb--Noskov \cite{AFN}, Beridze--Traczyk \cite{Ber-Tra2, Ber-Tra},  Calvez--Ito \cite{Calvez-Ito}, Datta \cite{datta}, Dlugie \cite{Dlugie},  Fullarton--Shadrach \cite{FullartonShadrach},  Gibson--Williamson--Yacobi~\cite{GWY}, and Witzel--Zaremsky \cite{Witzel-Zaremsky}.  The question of faithfulness for $\rho_4$ also appears as Question 3.1 in Margalit's problem list for mapping class groups \cite{Margalit:Problems}.  

\begin{maintheorem}
The Burau representation $\rho_4$ of $\B_4$ is faithful.  
\end{maintheorem}
As a first step to proving this result, we will give a new and simple topological proof of the faithfulness of the Burau representation in the case $n = 3$.  We will then adapt these ideas to the case of $n=4$ by exploiting the structure of point-pushing maps in $\B_4$.  

This approach is in contrast to earlier work on $\B_n$ for $n\geq 5$ and attempts to prove that the Burau representation $\rho_4$ is not faithful.  In particular, we do not consider elements in the image of $\rho_4$ that potentially generate a free group. We refer the reader to the second author's book \cite[Theorem 3.19]{birmanbook}, where this approach was first introduced.  In addition to the papers mentioned above, see also Bigelow~\cite[Section 3]{bigelow} and Moran~\cite{Moran} for details of other approaches.  

\p{Application to the Jones representation.}  In his seminal work, Jones introduced a representation of the braid group that contains the (reduced) Burau representation as a summand; for a full definition we refer the reader to his paper \cite{jones}.  As Jones points out, the faithfulness of Burau representation $\rho_n$ for any $n$ implies the faithfulness of the Jones representation of the braid group $B_n$. Therefore we immediately obtain the following corollary of our Main Theorem.
\begin{corollary}
The Jones representation of $\B_n$ is faithful for $n = 4$.
\end{corollary}
We also refer the reader to Kasahara's subsequent work on the Jones representation, in particular to his explanation of the equivalence between the faithfulness of the Burau representation and the Jones representation in the case $n = 4$ \cite[Remark 5.6]{kasahara}.

\p{Strategy of proof.}  Let $K_i$ denote the point-pushing subgroup of $\B_n$ obtained as the kernel of the Birman exact sequence for the disk $D_n$ by ``forgetting'' the $i^{th}$ marked point in the disk $D_n$; see Section~\ref{section:minimal position} for more details.  The intersection $\cap_{i=1}^n K_i$ of all point-pushing subgroups in $\B_n$ is known as the {\it Brunnian group} $\Brun_n$, and it is a normal subgroup of $\B_n$.   A theorem of Long states that the Burau represention $\rho_n$ is faithful on $B_n$ if it is faithful on any nontrivial noncentral normal subgroup of $\B_n$~\cite[Theorem 2.2]{long}.  Suppose now that $\ker(\rho_4)$ were nontrivial.  By Long's theorem, $\ker(\rho_4)$ intersects $\Brun_4$ nontrivially. 

\begin{proposition}
\label{prop:reduction to K_4}
    If the Burau representation $\rho_4$ is faithful on its restriction to the Brunnian group $\Brun_4$, then it is faithful on $\B_4$.
\end{proposition}

Building on ideas of Long--Paton and Bigelow, we associate to each braid in $\B_n$ a sequence of disks in $D_n$ that carry certain combinatorial data.  We use this technology to give a new proof of Magnus--Peluso's theorem that the Burau representation $\rho_3$ is faithful~\cite{magnuspeluso}.  Our proof yields the more general result that any $n$-strand braid whose associated disk sequence satisfies a certain {\it parity condition}, defined in Section~\ref{section:faithfulness for n=3}, does not lie in the kernel of $\rho_n$.  

Braids in $\B_n$ do not in general satisfy this parity condition when $n \geq 4$.  However, we show that when $n = 4$, the parity condition ``almost'' holds for any point-pushing braid $\Phi \in K_4$ that admits a factorization as a {\it proper product} of certain push-maps.  The next step is to ``correct'' $\Phi$ by embedding $K_4$ in $\B_5$, that is, we use $\Phi$ to construct a 5-braid satisfying the parity condition.  We then apply a result of Moody to conclude that the original braid $\Phi \in K_4$ does not lie in the kernel of $\rho_4$.  The final step is to show that any braid in $K_4$ is conjugate to another braid in $K_4$ that can be realized as a proper product of push-maps.  

Finally, we remark that, {\it a priori}, it is not necessarily clear that $\Brun_n$ is nontrivial.  The simplest example of a Brunnian braid is a three-strand braid that closes to form the Borromean rings.  Figure~\ref{figure:brunnian} gives an example of a nontrivial element of $\Brun_4$, the case of interest to us, and a similar construction yields nontrivial examples of Brunnian braids in $\B_n$ for any $n$. Moreover, Whittlesey has shown that all nontrivial Brunnian braids are pseudo-Anosov \cite{kw}.  We note that Whittlesey's results are stated for the mapping class group of the $(n+1)$-punctured sphere, but the result holds in $\B_n$; see Lee-Song's discussion of this point \cite[Section 1]{LeeSong}.   
\begin{figure}[htpb!]
    \centering

    \includegraphics[width=.4\linewidth]{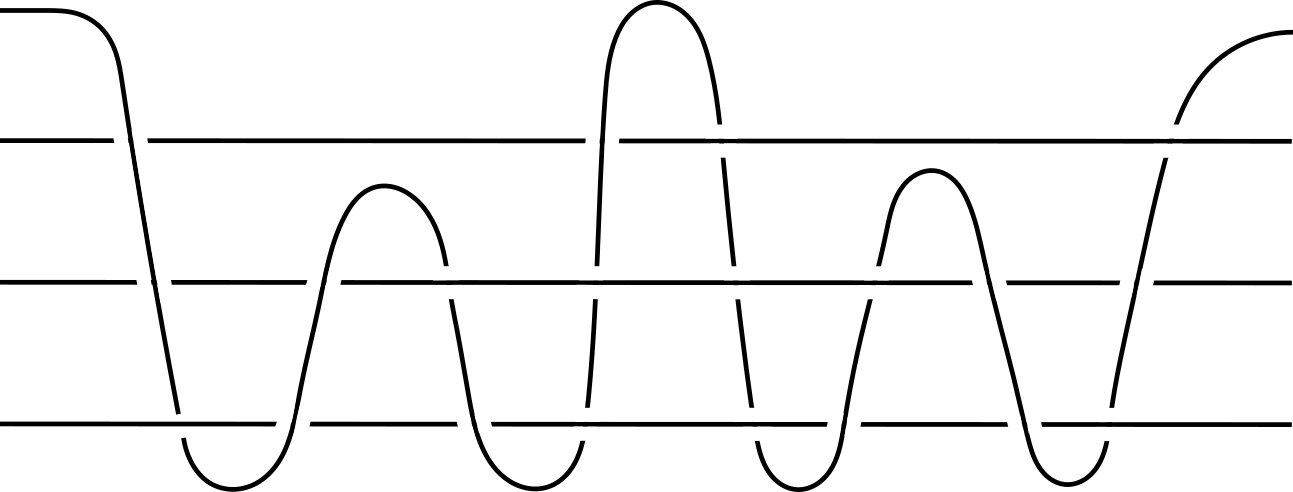}
        \caption{A Brunnian 4-braid $\beta$.}
    \label{figure:brunnian}
\end{figure}

\p{Kernel and image of Burau.}  There is a great deal of literature on the faithfulness of the Burau representation, and this paper settles that question in the final outstanding case.  The questions of giving both a useful characterization of its kernel and of its image remain wide open in general, and the authors believe that the methods of this paper may be useful in achieving progress on these important questions.  

\p{Outline of the paper.}  
In Section~\ref{section:Moody and CP}, we introduce the Moody polynomial of a braid, an invariant that provides an obstruction for an element to lie in the kernel of the Burau representation.  Building on Moody's ideas and their development by Long--Paton, Bigelow later introduced combinatorial tools that are useful for computing the obstruction; we describe these in Section~\ref{section:disk sequences}.  Using these tools, in Section~\ref{section:faithfulness for n=3}, we give our new topological proof of faithfulness of the Burau representation in the  case $n=3$.  We introduce our key technical tool, {\it proper products} of point-pushing braids, in Section~\ref{section:minimal position}.  In Section~\ref{section:disks in B_4} we construct point-pushing maps arising as particular proper products, and use these to establish that the Burau representation $\rho_4$ is faithful on its restriction to the point-pushing group $\Brun_4$; by Proposition~\ref{prop:reduction to K_4}, this completes the proof of the Main Theorem.  Finally, we give an example illustrating the constructions used in our proof in Section~\ref{section:appendix}.

{\bf Acknowledgements.} The authors heartily thank Dan Margalit for many helpful discussions and are especially grateful for his extensive and valuable feedback on initial drafts of this paper.  
%His comments and insights have substantially improved our paper and led us to a more intuitive and streamlined version of our proof.  
The second author would also like to thank him for his interest in this project from its inception seven years ago, when the first two authors had begun to work on it together.  The authors also thank Khalid Bou-Rabee, Michael Dougherty, and Benson Farb for additional helpful comments on earlier drafts and Ken Birman for further helpful remarks. The first author thanks David Gabai for his invaluable mentorship and constant encouragement. The third author is also grateful to the Institute for Computational and Experimental Research in Mathematics (ICERM) at  Brown University, where in 2022 she was a participant in the semester-long program ``Braids'', during which initial discussions that led to her joining this project took place, and to Muffy Calder for her support and encouragement throughout.  AI was used to proofread a draft of the manuscript.

%%%%%%%%%%%%%%
%%%%%%%%%%%%%%
\section{The Moody polynomial and winding numbers}
\label{section:Moody and CP}
 While Burau originally defined his representation by giving its values on the standard Artin generators of $\B_n$ (see, for example, \cite[Section 4.2]{birmanbrendle}), it will be more useful for present purposes to define $\rho_n$ via the action of $\B_n$ on the disk $D_n$, which we consider here as an $n$-times punctured disk.  Let $p_*$ denote a point in the boundary of the disk $D_n$, and let $x_1, \dots, x_n$ denote the standard free generators of $\pi_1(D_n, p_*)$.  We consider the mapping $\pi_1(D_n, p_*) \to \Z$ that takes a word in $x_1, \dots x_n$ to the sum of its exponents, and let $\widetilde{D_n}$ denote the covering space associated with its kernel.  

The group of covering transformations of $\widetilde{D_n}$ is isomorphic to $\mathbb{Z}$, which we denote as a multiplicative group generated by $t$. Following the treatment of Long--Paton, it will be convenient for us to work with the relative homology group $H_1(\widetilde{D_n}, \{\widetilde{p_*}\})$, where $\{\widetilde{p_*}\}$ denotes the full pre-image of the basepoint $p_*$ in the cover $\widetilde{D_n}$.  The braid group $\B_n$ acts on the $\mathbb{Z}[t, t^{-1}]$-module $H_1(\widetilde{D_n}, \{\widetilde{p_*}\})$, which is free of rank $n$; this action is the {\it unreduced Burau representation} $\rho_n$, which we will refer to simply as the Burau representation.  

We next record a basic fact that follows directly from the definition of $\rho_n$.
\begin{observation}
\label{observation:basic}
    Let $f: \B_n \rightarrow \B_{n+1}$ be the inclusion map corresponding to the standard embedding of $D_n$ into $D_{n+1}$.  If $\Phi \in \B_n$ lies in the kernel of $\rho_n$, then $f(\Phi)$ lies in the kernel of $\rho_{n+1}$.
\end{observation}
Observation~\ref{observation:basic} will be crucial in our proof of the Main Theorem.

\begin{figure}[htpb!]
\centering
    \begin{overpic}[width=0.8\linewidth]{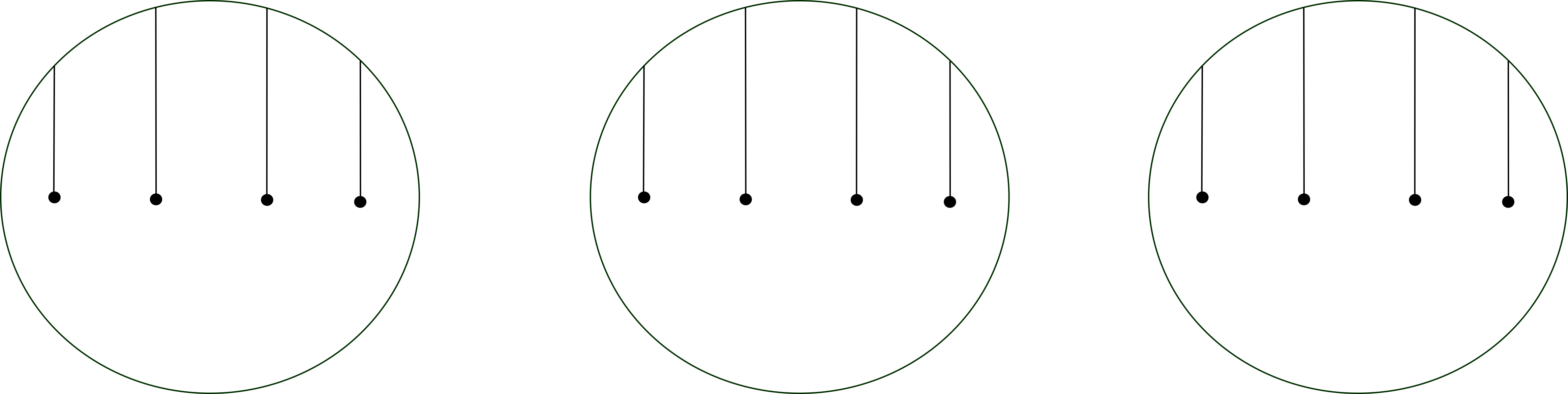}
    \put(-7,14){$\dots$}
    
    \put (1.5,15.5){\tiny $+$} 
    \put (4,15.5){\tiny $-$} 
    \put (8,15.5){\tiny $+$} 
    \put (10.7,15.5){\tiny $-$} 
    \put (14.8,15.5){\tiny $+$} 
    \put (17.7,15.5){\tiny $-$} 
    \put (21,15.5){\tiny $+$} 
    \put (23.2,15.5){\tiny $-$} 

    \put (39,15.5){\tiny $+$} 
    \put (42,15.5){\tiny $-$} 
    \put (44.9,15.5){\tiny $+$} 
    \put (48,15.5){\tiny $-$} 
    \put (52,15.5){\tiny $+$} 
    \put (55,15.5){\tiny $-$} 
    \put (58,15.5){\tiny $+$} 
    \put (61,15.5){\tiny $-$} 

    \put (74.5,15.5){\tiny $+$} 
    \put (77,15.5){\tiny $-$} 
    \put (80.9,15.5){\tiny $+$} 
    \put (83.5,15.5){\tiny $-$} 
    \put (88,15.5){\tiny $+$} 
    \put (91,15.5){\tiny $-$} 
    \put (94,15.5){\tiny $+$} 
    \put (96.5,15.5){\tiny $-$} 

    \put(12, -3) {$t^0$}
    \put(49, -3) {$t^1$}
    \put(86,-3) {$t^2$}

    \put(103,14){$\dots$}

    \put(3, 10){$p_1$}
    \put(9, 10){$p_2$}
    \put(16, 10){$p_3$}
    \put(22, 10){$p_4$}

    \put(40, 10){$p_1$}
    \put(47, 10){$p_2$}
    \put(53, 10){$p_3$}
    \put(59, 10){$p_4$}

    \put(76, 10){$p_1$}
    \put(82, 10){$p_2$}
    \put(88, 10){$p_3$}
    \put(94, 10){$p_4$}
    
    \end{overpic}
    \vspace{2mm}
    \caption{The zeroth, first, and second ``decks'' in the universal cyclic cover $\widetilde{D_4}$ of the disk $D_4$. } 
    \label{figure:cover}
\end{figure}
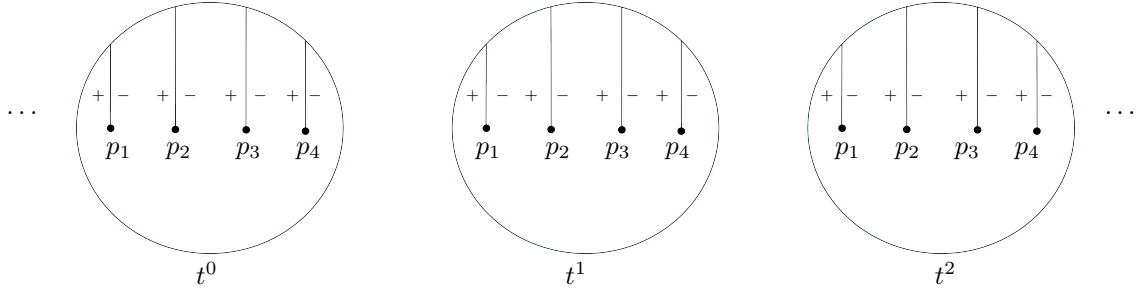

{\bf Moody polynomials.} In what follows, we will closely follow Bigelow's notation and exposition \cite{bigelow}. We will use the term {\it arc} to refer to a proper embedding of an interval in the disk $D_n$ and {\it subarc} to refer to its restriction to a compact connected subinterval.  Our arcs will normally be oriented, and it will often be convenient to consider arcs in terms of their image in the disk $D_n$ rather than as functions per se. 

Let $\alpha$ and $\beta$ be two oriented arcs in $D_n$, considered as a disk with marked points, whose endpoints lie in the set $\{ p_1, \ldots, p_n, p_* \}$, and choose two corresponding lifts $\tilde{\alpha}$ and $\tilde{\beta}$ in $\widetilde{D_n}$.  We define the {\it Moody polynomial  $\Moody(\alpha, \beta) \in \Z [t, t^{-1}]$  of the oriented arcs $\alpha$ and $\beta$} as follows: 
\[
\Moody(\alpha, \beta) = \sum_{\ell \in \Z} (t^\ell \cdot \tilde{\alpha}, \tilde{\beta}) t^\ell
\]
where $(t^\ell \cdot \tilde{\alpha}, \tilde{\beta})$ denotes the algebraic intersection number of the lift $t^\ell \cdot \tilde{\alpha}$ with $\tilde{\beta}$ in the cover $\widetilde{D_n}$.  We note that the algebraic intersection number of arcs does not include intersections at any shared endpoints; in our case, we will always choose our arcs so that the endpoints of $\alpha$ are disjoint from those of $\beta$.  We also note that the Moody polynomial $\Moody(\alpha, \beta)$ is only well defined up to multiplication by a power of $t$, due to its dependence on our choice of the lifts $\tilde{\alpha}$ and $\tilde{\beta}$.  Our convention will be to choose our lifts so that the first point at which $\beta$ crosses $\alpha$ is assigned the monomial $\pm t^0$.

Following Bigelow \cite{bigelow} and Long--Paton \cite{Long-Paton}, we now describe a planar interpretation of the Moody polynomial that enables us to work entirely in the base space $D_n$ rather then in the cover.  Without loss of generality, we can assume that the oriented arcs $\alpha$ and $\beta$ intersect transversely in finitely many points; we label these $q_1, \ldots, q_m$ according to the order in which they appear as we traverse the arc $\beta$ according to its orientation.  Each intersection point $q_i$ corresponds to a single point of intersection between two lifts $\tilde{\alpha}$ and $\tilde{\beta}$, and hence each point $q_i$ contributes a monomial $\pm t^{k_i}$ to the overall sum in the Moody polynomial.  Here the exponent $k_i$ is such that $\tilde{\beta}$ and $t^{k_i} \cdot \tilde{\alpha}$ cross at a lift of $q_i$, and the sign $\epsilon_i$ of the monomial is the sign of that crossing where $\epsilon_i \in \{ -1, 1 \}$.  Then an equivalent formulation of the Moody polynomial of the arcs $\alpha$ and $\beta$ is as follows:
\[
 \Moody(\alpha, \beta) = \sum_{i = 1}^m \epsilon_i t^{k_i}.
\]
We emphasize that it may happen that $k_i = k_j$ for $i \neq j$.  Indeed, our business in this paper will be to determine conditions under which this could happen.  When we write the Moody polynomial as the sum of $m$ terms of the form $\epsilon_i t^{k_i}$ without combining like terms, we will refer to this as the {\it unsimplified Moody polynomial}; when we combine all like terms we will refer to this as the {\it simplified Moody polynomial}.  
Moody showed that the polynomial $\Moody(\alpha, \beta)$ completely encodes the faithfulness of the Burau representation \cite{moody}. We will use one direction of his characterization; see Theorem~\ref{theorem:moody2}.

{\bf Conventions and color-coding.} For the remainder of this paper, we will always choose the arc $\alpha$ to be the horizontal arc joining $p_1$ to $p_2$ oriented from left to right, and we will color it blue in all figures, as shown in Figure~\ref{figure:std-arcs}.  We also let $\beta_*^3$ denote the oriented arc shown in Figure~\ref{figure:std-arcs} from the basepoint $p_*$ to the puncture $p_3$ and color it red.  
\begin{figure}[htpb!]
    \centering
    \begin{overpic}[scale=1.7]{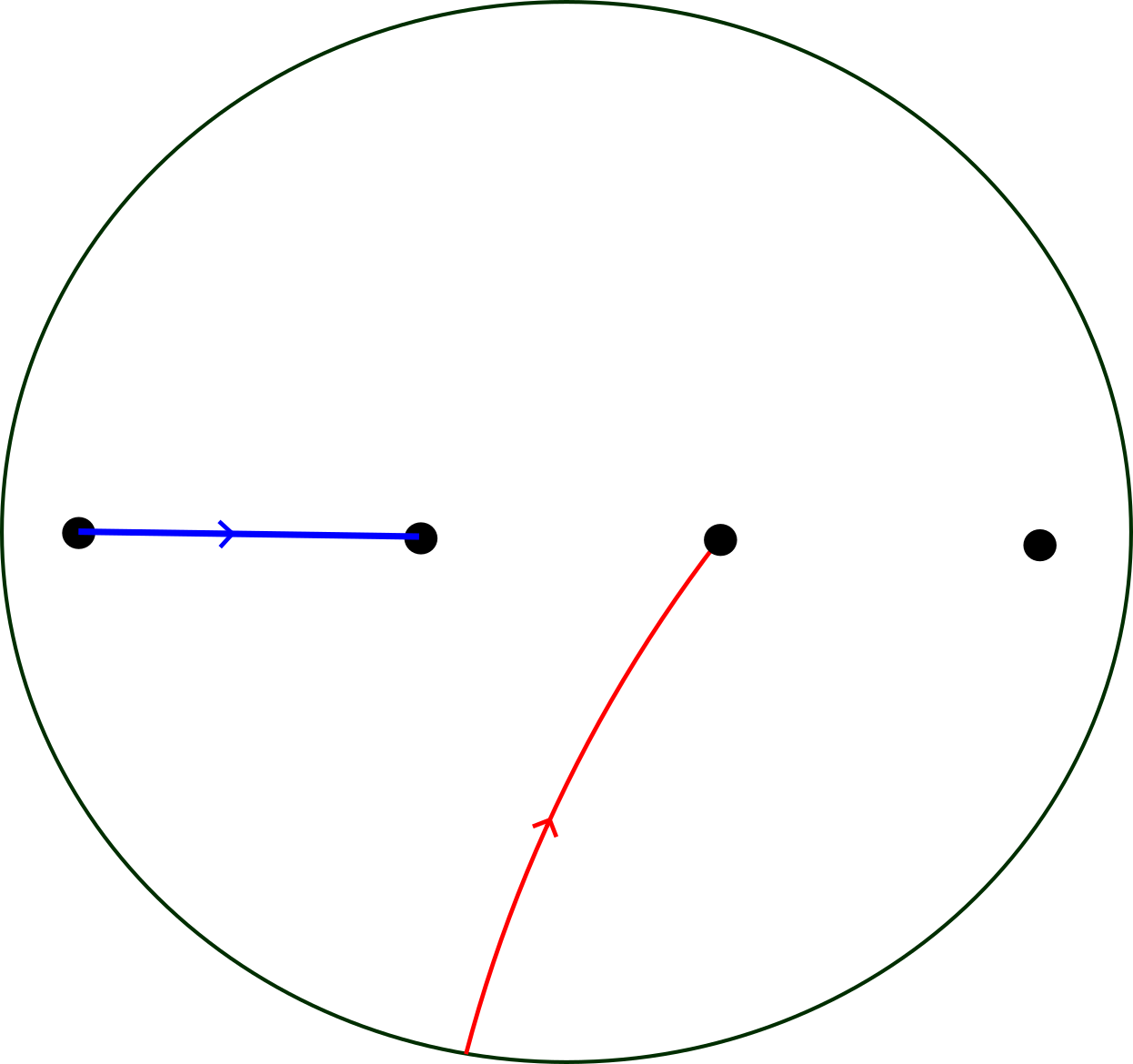}
        \put(22,40){\textcolor{blue}{$\alpha$}}
        \put(50,20){\textcolor{red}{$\beta_*^3$}}
        \put(38,-5){$p_*$}
        \put(9,53){$p_1$}
        \put(35,53){$p_2$}
        \put(60,53){$p_3$}
        \put(85,53){$p_4$}
    \end{overpic}
    \vspace{2mm}
    \caption{The blue  arc $\alpha$ joins $p_1$ to $p_2$. The red arc $\beta_*^3$ joins the basepoint on $\partial D$ to the point $p_3$. }
      \label{figure:std-arcs}
\end{figure}
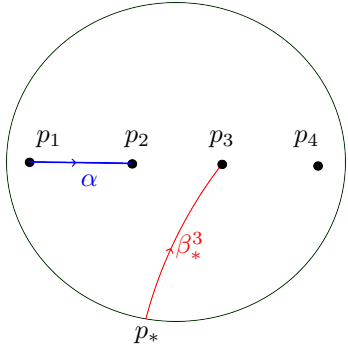

When we apply a braid $\Phi \in \B_n$ to the arc $\beta_*^3$, we will usually denote this by $\beta = (\beta_*^3) \Phi$ and also color $\beta$ red; note that we write $\Phi$ on the right of the arc to which we are applying it because we will follow the tradition of writing multiplication of braids from left to right.  

Furthermore, from this point on, a blue (respectively red) arc will always be a subarc of $\alpha$  (respectively $\beta$). Later we will add gold to our list of special colors.  Our sign convention will be to say that $q_i$ is a positive crossing, that is, $\epsilon_i = 1$, if $\beta$ is directed downwards at $q_i$ and that $q_i$ is a negative crossing with $\epsilon_i = -1$ if $\beta$ is directed upwards at $q_i$.

For any arcs $\gamma$ and $\delta$ in the disk $D_n$, we let $\iota(\gamma, \delta)$ denote their geometric intersection number.  Consider now our fixed choice of arc $\alpha$.  A remarkable result of Moody states that the non-faithfulness of the Burau representation is equivalent to the existence of an arc $\beta$ joining the basepoint $p_*$ to the marked point $p_3$ in $D_n$ such that $\iota(\alpha, \beta) > 0$ and $\Moody(\alpha, \beta) = 0$ \cite{moody}.  With this in mind, we define the {\it Moody polynomial of the braid} $\Phi \in \B_n$ as follows: 
\begin{align*}
\Moody_{\Phi} = \Moody(\alpha, (\beta_*^3)\Phi ). 
\end{align*}
We remark that by taking the boundary of a regular neighborhood of the arc $(\beta_*^3)\Phi$, we can view as a loop based at $p_*$ traveling around $p_3$, and that $\Moody_{\Phi}$ is well defined on the homology class of a lift of $[(\beta_*^3)\Phi]$ in  $H_1(\widetilde{D_n}, \{\widetilde{p_*}\})$, which we denote by $[(\beta_*^3)\Phi]$.   We also emphasize that an element $\Phi \in \B_n$ uniquely determines an arc $\beta = (\beta_*^3)\Phi$, but the converse does not hold: there exist infinitely many pairs $\Phi_1, \Phi_2 \in \B_n$ for which $(\beta_*^3) \Phi_1 = (\beta_*^3) \Phi_2$.  We refer the reader to Long--Paton~\cite[Section 1]{Long-Paton} and Bigelow~\cite[Theorem 1.4]{bigelow} for further details.  For our purposes, we do not need the full strength of Moody's result; we will only require one direction of the equivalence. 
\begin{theorem}[Moody]
\label{theorem:moody0}
Let $n \geq 3$.  If, given any oriented arc $\beta$ from $p_*$ to $p_3$ such that $\iota(\alpha, \beta) > 0$, we have that $\Moody(\alpha, \beta) \neq 0$, then the Burau representation $\rho_n$ is faithful. Equivalently, if we have that $\Moody_\Phi \neq 0$, for any braid $\Phi$ such that $\Phi(p_3) = p_3$ and $\iota(\alpha, (\beta_3^*)\Phi) > 0$, then the Burau representation $\rho_n$ is faithful. 
\end{theorem}
The equivalence of the two statements follows from the fact that every oriented arc $\beta$ from $p_*$ to $p_3$ arises as $\beta = (\beta_*^3) \Phi$ for some $\Phi \in \B_n$ by the change of coordinates principle; note that such a braid $\Phi$ necessarily fixes the point $p_3$.  It will be useful for our purposes to adapt Moody's theorem to obtain a specific obstruction to an individual braid lying in the kernel of $\rho_n$, for which we give a short proof for the sake of completeness.

\begin{theorem}[Moody]
\label{theorem:moody2}
    Let $\Phi \in \B_n$ for $n \geq 3$.  If $\Moody_{\Phi\cdot \Gamma} \neq \Moody_{\Gamma}$ for some $\Gamma\in B_n$, then $\Phi$ does not lie in the kernel of $\rho_n$. 
\end{theorem}

\begin{proof}
    Suppose that $\Phi$ lies in the kernel of $\rho_n$.  Then $\Gamma^{-1}\cdot \Phi \cdot \Gamma$ must also lie in the kernel.  We have that $(\beta_*^3)\cdot \Gamma$ corresponds to an element $[(\beta_*^3)\cdot \Gamma]$ in the homology of the cover $\widetilde{D_n}$ relative to the pre-image of a basepoint.  Since $\Gamma^{-1}\cdot \Phi\cdot \Gamma$ acts trivially on this homology group, we have that $[(\beta_*^3) \cdot (\Phi\cdot \Gamma)]= [((\beta_*^3) \cdot \Gamma) \cdot (\Gamma^{-1}\cdot \Phi\cdot \Gamma)] = [(\beta_*^3) \cdot \Gamma]$.  The result now follows from our previous observation that $\Moody_\Phi$ well defined on the homology class $[(\beta_*^3)\Phi]$ in   $H_1(\widetilde{D_n}, \{\widetilde{p_*}\})$. 
    \end{proof}

 In Section~\ref{section:disks in B_4} we will identify certain products of braids in $\B_4$ to which we will then apply Theorem~\ref{theorem:moody2} in order to establish the faithfulness of $\rho_4$.  Moreover, as discussed in Section~\ref{section:intro}, we ultimately reduce our search to the Brunnian subgroup $\Brun_n$ in $\B_n$, where all nontrivial braids are pseudo-Anosov, and hence for all $\Phi \in \Brun_n$ we have that $\iota(\alpha,(\beta^3_*)\cdot\Phi) > 0$.  We will also use the following special case of Theorem~\ref{theorem:moody2}.
\begin{corollary}
\label{corollary:kernel element}
    Let $\Phi \in \B_n$ for $n \geq 3$.  If $\Moody_\Phi \neq 0$, then $\Phi$ does not lie in the kernel of $\rho_n$. 
\end{corollary}

\section{Disk sequences and total winding number} 
\label{section:disk sequences}
Let $\beta = (\beta_*^3)\Phi$, and as above, we label the points of intersection of $\alpha$ with $\beta$ by $q_1, \ldots, q_m$ according to the order in which we encounter them as we travel along $\beta$.  For $i = 1, \ldots m-1$, we let $\alpha_i$ and $\beta_i$ denote the subarcs of $\alpha$ and $\beta$, respectively, that join the intersection point $q_i$ to $q_{i+1}$.  Note that, with our choice of notation, this means that each arc $\beta_i$ contains no intersection point $q_j$ in its interior, while an arc $\alpha_i$ may contain a number of other intersection points in its interior.  

By construction, for each $i = 1, \ldots m-1$, we have that $\alpha_i \cup \beta_i$ bounds a $k$-punctured disk $\Delta_i$.  We define the {\it total winding number} $W_i$ associated to the disk $\Delta_i$ to be equal to $k$ if $\beta_i$ is oriented clockwise with respect to $\Delta_i$, or to be equal to $-k$ otherwise.  We will refer to the sequence $W_1, \ldots, W_{m-1}$ as the {\it winding number sequence} of $\Phi$, and to the sequence $\Delta_1, \Delta_2, \ldots, \Delta_{m-1}$ as the {\it disk sequence} of $\Phi$; see Figure~\ref{figure:cp-examples} for examples.  

We will consider two such disks $\Delta_i, \Delta_j$ to be {\it equivalent} if there is an isotopy of $(D_n, \alpha)$ taking the $\beta$-component of $\partial \Delta_i$ to the the $\beta$-component of $\partial \Delta_j$; in other words  $\alpha$ is fixed setwise and the endpoints of the $\beta$-arc must be contained in $\alpha$ at each level of the isotopy).  We note that this notion of equivalence corresponds to allowing isotopies of the arc $\beta$ in the disk $D_n$.  

We will see that the combinatorial data of the winding number sequence is sufficient for determining the faithfulness of the Burau representation in the case $n = 3$, and that the more detailed information of the disk sequence is required in the case $n = 4$.  The following lemma appears as a remark in Bigelow's paper \cite[Section 3]{bigelow}.  His remark describes the relationship between winding numbers and the degree $k_i$ of the monomial that each intersection point $q_i$ contributes to the Moody polynomial.
     \begin{figure}[htpb!]

        \centering
        \begin{subfigure}{.45\linewidth}\centering
                    \includegraphics[width=0.76\linewidth]{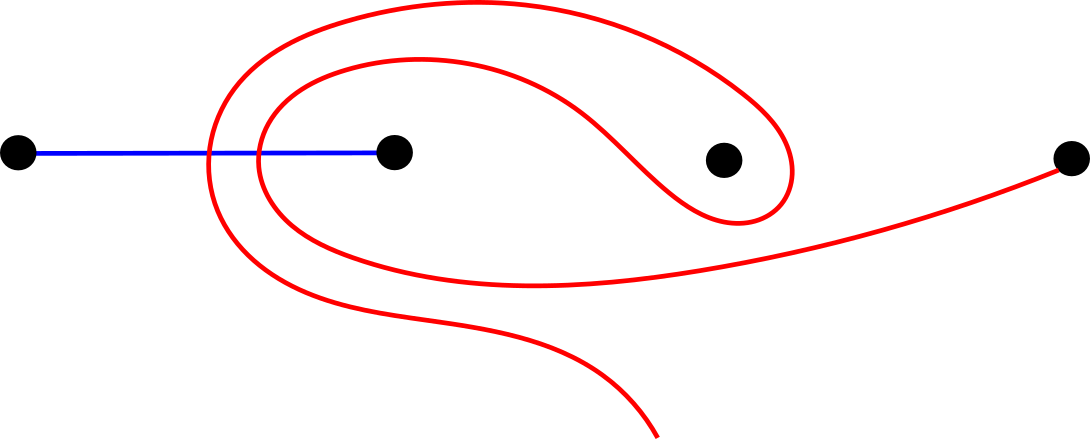}
        \caption{}
        \label{figure:cp-example-1}
        \end{subfigure}
        \begin{subfigure}{.5\linewidth}\centering
                    \includegraphics[width=0.7\linewidth]{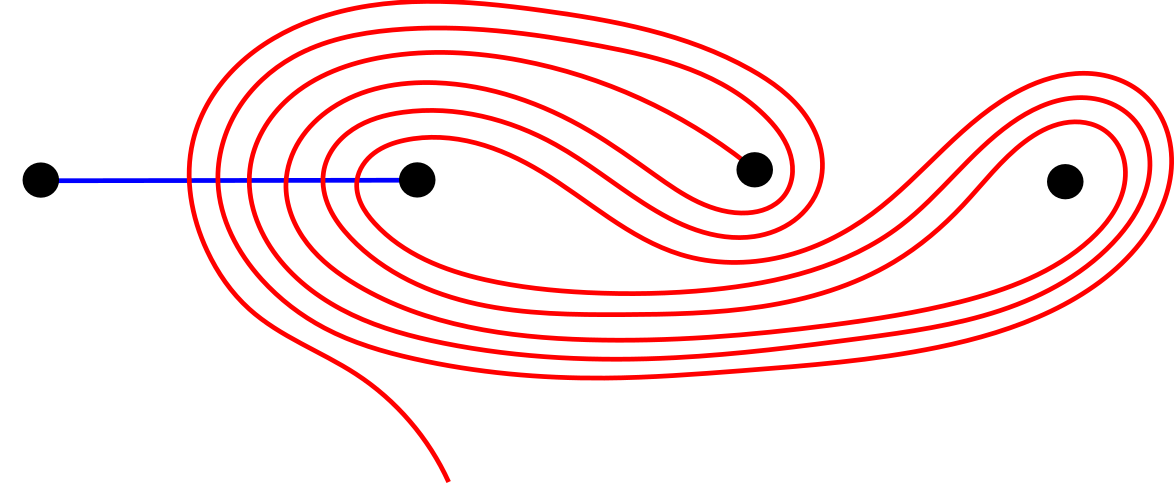}
           \caption{}
                % \label{figure:cp-example-2}
        \end{subfigure}
        \caption{The arcs $\alpha$ and $\beta$ in (a) give rise to a singleton winding number sequence given by $W_1=1$, with corresponding Moody polynomial $1-t$. The winding number sequence for the arcs in (b) is given by $1, 1, -1, -2, 1$; its Moody polynomial is $2 - 2t + t^2 - t^{-1}$. }
        \label{figure:cp-examples}
    \end{figure}

\begin{lemma}[Bigelow \cite{bigelow}] 

\label{lemma:observation:planarpoly}
Let $q_i$ and $q_{i+1}$ be two points in $\alpha \cap \beta$ that are consecutive with respect to the oriented arc $\beta$, and let $k_i$ and $k_{i+1}$ be the exponents of the monomials corresponding to $q_i$ and $q_{i+1}$.  If $W_i$ denotes the $i ^{th}$ term in the winding number sequence for $\beta$, then we have:
\begin{equation}  
W_{i}=k_{i+1} - k_i.
\end{equation}
\end{lemma}
 
We record here the following useful corollary of Lemma~\ref{lemma:observation:planarpoly}, also observed by Bigelow:
\begin{corollary}
\label{corollary:exponentsequal}
Let $W_1, \ldots, W_m$ be a winding number sequence for a braid in $B_n$, and let $k_i$ and $k_j$ be the exponents of the Moody monomials corresponding to two points $q_i$ and $q_j$ in $\alpha \cap \beta$.  Then $k_i = k_j$  if and only if 
$\sum_{k=i}^{j-1} W_k = 0$.
\end{corollary}

Equipped with this useful criterion for tracking repeats of an exponent occurring in the Moody polynomial, we will proceed in the next section to consider the question of faithfulness of $\rho_n$ when $n=3$.

%%%%%%%%%%%%%%
%%%%%%%%%%%%%%
\section{Faithfulness for three strands} 
\label{section:faithfulness for n=3}

The following theorem was first proved by Magnus and Peluso in 1969 using purely algebraic methods.  We give a new proof. 

\begin{theorem}
[Magnus--Peluso \cite{magnuspeluso}]
\label{theorem:Burau3}
The Burau representation $\rho_3$ is faithful.
\end{theorem}

\begin{proof}Using the same notation as in the previous section, we let $\Phi \in \B_3$ be an element such that the geometric intersection of $\alpha$ and $\beta = (\beta_*^3)\Phi$ is equal to $m > 0$. By Theorem~\ref{theorem:moody0}, it suffices to show that any single term of the simplified Moody polynomial $\Moody_{\Phi} = \sum_{i = 1}^m\epsilon_i t^{k_i}$ is nonzero.
If $m = 1$, then the Moody polynomial consists of a single linear term, and hence is nonzero.  Suppose now that $m > 1$, and let $W_1, \ldots, W_{m-1}$ denote the winding number sequence of the braid $\Phi$. Each corresponding disk is bounded by a subarc of $\alpha$ and a subarc of $\beta$ and contains 1, 2, or 3 marked points.  The  assumption that $\alpha$ and $\beta$ are in minimal position then implies that, up to reflection about the coordinate axes and up to homeomorphism of the disk preserving each marked point $p_i$ pointwise and the arc $\alpha$ setwise, each disk in the disk sequence is equivalent to one of the three types of regions illustrated in Figure~\ref{figure:WN-disks for $n=3$}, possibly after applying a symmetry.
\begin{figure}[htpb!] 
    \centering
    \begin{subfigure}{0.29\textwidth}\centering
    \includegraphics[scale=1.7]{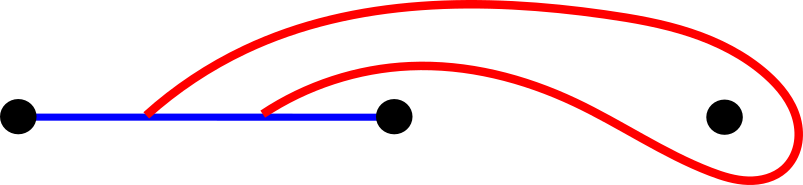}
    \caption{$|W_i|=1$}
    \end{subfigure}
    \begin{subfigure}{0.4\textwidth}\centering
    \includegraphics[scale=1.7]{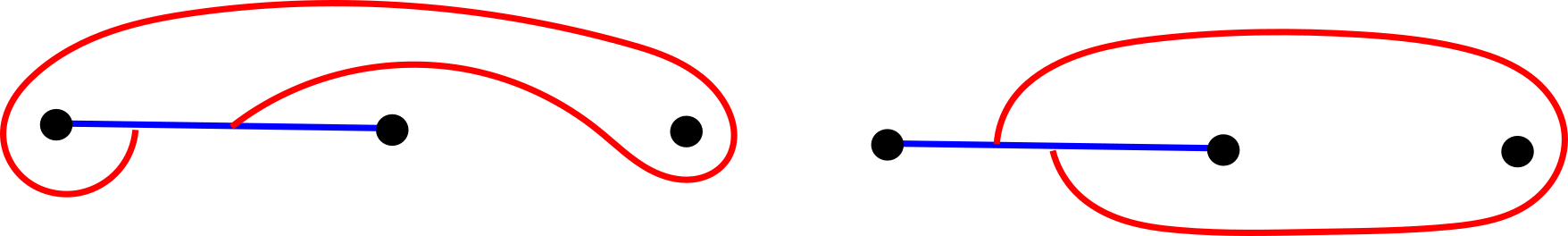}
    \caption{$|W_i|=2$}
    \end{subfigure}
    \begin{subfigure}{0.29\textwidth}\centering
    \includegraphics[scale=1.7]{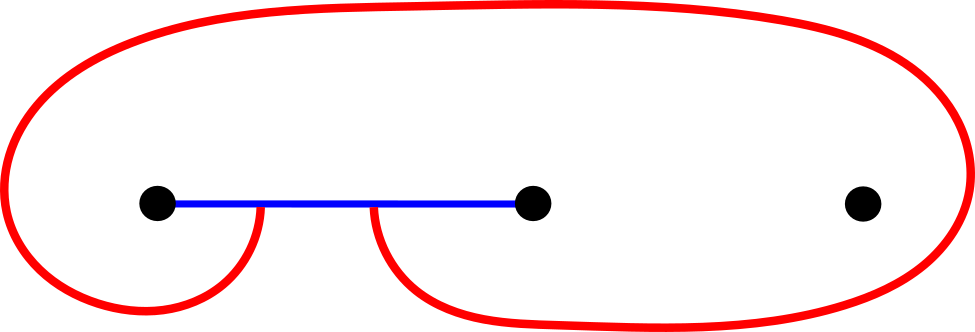}
    \caption{$|W_i|=3$}
    \end{subfigure}
    \caption{Possible disks in the disk sequence of an element of $\B_3$.}
    \label{figure:WN-disks for $n=3$}
\end{figure}

Notice that for each $i$ such that $|W_i|=2$, the sign $\epsilon_i$ of the two points of intersection of $\beta$ with $\alpha$ must be the same, that is, $\epsilon_i = \epsilon_{i+1}$.  Furthermore, if $|W_i|\in \{1,3\}$, then we have that $\epsilon_i = -\epsilon_{i+1}$. 

Suppose now that $k_i = k_j$ for two distinct terms in the unsimplified Moody polynomial.  Then $\sum_{\ell=i}^{j-1} W_\ell=0$ by Corollary~\ref{corollary:exponentsequal}.   Hence for $i\leq \ell \leq j-1$ we must have an even number of terms $W_\ell$ with $|W_\ell|\in \{1,3\}$. By our observations in the previous paragraph, it follows that $\epsilon_i = \epsilon_j$.  In other words, no cancellation among terms of degree $k_i$ can occur in the summation above. It follows that the Moody polynomial of $\Phi$ is not equal to zero, and therefore the Burau representation of $B_3$ is faithful.
\end{proof}

\p{The parity condition.}  We note that the only information that we used in the proof of Theorem~\ref{theorem:Burau3} was the fact that disks containing an odd number of punctures correspond to a change of sign of corresponding coefficients in the Moody polynomial, while disks containing an even number of punctures preserve the sign of the corresponding coefficients.  We will record this key observation as a lemma. 

To make this precise, let $q_i$ be a point of intersection of the oriented arcs $\alpha$ and $\beta$.  As in Section~\ref{section:Moody and CP}, we let $\epsilon_i$ denote the algebraic intersection number $\hat{i}(\beta, \alpha)$ at $q_i$.  If $\Delta$ is a disk in the sequence for $\Phi$ corresponding to the pair of points $q_i, q_{i+1} \in \alpha \cap \beta \in \partial \Delta$, then we will say that $\Delta$ is {\it sign-changing} if $\epsilon_i = - \epsilon_{i+1}$ and {\it sign-preserving} otherwise. Let $\mcP(\Delta)$ denote the set of marked points in $D_n$ that are contained in the interior of $\Delta$.  We will say that a braid $\Phi$ satisfies the {\it parity condition} if a disk $\Delta$ in its disk sequence is sign-changing if and only if $|\mcP(\Delta)|$ is odd.  

In other words, our proof of Theorem~\ref{theorem:Burau3} reduces to showing that every 3-strand braid satisfies the parity condition. Our next lemma shows that the parity condition gives a sufficient criterion for a braid $\Phi \in \B_n$ to not be contained in the kernel of the Burau representation $\rho_n$. 

Recall from Section~\ref{section:Moody and CP} that the Moody polynomial of a braid $\phi$ is given by 
    \begin{align*}
        \Moody_{\phi} = \sum_{i = 1}^{m}\epsilon_i t^{k_i}
    \end{align*}
    where each $\epsilon_it^{k_i}$ is the monomial associated to the $i$-th crossing of $\alpha$ and $\beta=(\beta_*^3)\phi$, and where $m = \iota(\alpha, \beta)$. We say that the Moody polynomial of $\phi$ {\it admits no cancellations} if there is no pair of indices $i , j$ such that $k_i = k_j$ and $\epsilon_i = -\epsilon_j$ in the above expression.   

\begin{lemma}\label{lemma:no-cancellations}
    Let $\phi\in B_n$.  If $\phi$ satisfies the parity condition, then the Moody polynomial of $\phi$ admits no cancellations. 
\end{lemma}

\begin{proof}
    By Corollary \ref{corollary:exponentsequal}, if $k_i=k_j$ then $|\mcP(\Delta_i)|+\dots+|\mcP(\Delta_{j-1})|=0$. In particular this sum is even. However, by the assumption that a given disk in the disk sequence is sign-changing if and only if it contains an odd number of punctures, this means there are an even number of sign-changing disks in the subsequence $\Delta_i,\dots,\Delta_{j-1}$. In particular we have $\epsilon_i=\epsilon_j$ and so this completes the proof.
\end{proof} 

In general, $n$-strand braids do not satisfy the parity condition when $n \geq 4$.  Our strategy for proving the Main Theorem will be to identify certain 4-braids satisfying the parity condition.  

\section{Push-maps and minimal position}  
\label{section:minimal position}

As discussed in Section~\ref{section:intro}, in order to prove our Main Theorem, it suffices to establish that $\rho_4$ is faithful when restricted to the subgroup of $B_4$ consisting of point-pushing maps, also known simply as push-maps.   In light of Lemma~\ref{lemma:no-cancellations}, our strategy will be to identify elements of this subgroup satisfying the parity condition. In this section, as a first step in that direction, we will characterize configurations of the various arcs and curves involved that lead to the creation of bigons between $\alpha$ and the image of $\beta_*^3$ under the action of push-maps. 

For the purpose of defining push-maps, it is convenient to use the notation $\Mod(D_n)$ to denote the braid group $\B_n$.  Let $p \in D_n$ be one of the marked points in $D_n$, and let $\Mod(D_n, p)$ denote the subgroup of $\Mod(D_n)$ that fixes $p$.   There is a forgetful map $\Mod(D_n, p) \to \Mod(D_{n-1})$, and the Birman exact sequence for $D_n$ identifies the kernel of this map with 
$\pi_1(D_{n-1},p)$:
\begin{align*}
1\xrightarrow{} \pi_1(D_{n-1}, p) \xrightarrow{Push} \Mod(D_n, p) \xrightarrow{Forget} \Mod(D_{n-1}) \xrightarrow{} 1
\end{align*}
Given a loop $\Gamma$ in $\pi_1(D_{n-1},p)$, we can consider the {\it push-map} obtained by pushing the point $p$ along $\Gamma$, and the image of $\pi_1(D_{n-1},p)$ in $\Mod(D_n, p) \subset \Mod(D_n)$ is an example of a {\it point-pushing subgroup}.  We refer the reader to Farb-Margalit's book for a detailed discussion of point-pushing maps and subgroups~\cite[Section 4.2.2]{primer} .   

Now, let $K_i \cong \pi_1(D_{n-1}, p_i)$ denote the point-pushing subgroup of $\B_n$ corresponding to the marked point $p_i$ in the disk $D_n$; see Figure~\ref{figure:cp-example} for an example of an element of the point-pushing group $K_4$ in $\B_4$ and its effect on the arc $\beta_*^3$.  We remind the reader that throughout the paper, our figures show the arc $\alpha$ in blue and arcs of the form $\beta = \Phi(\beta_*^3)$ in red.  Henceforth we will also show any loops representing push-maps in gold.  

\begin{figure}[htpb!]
        \centering
        \begin{overpic}[width=\linewidth]{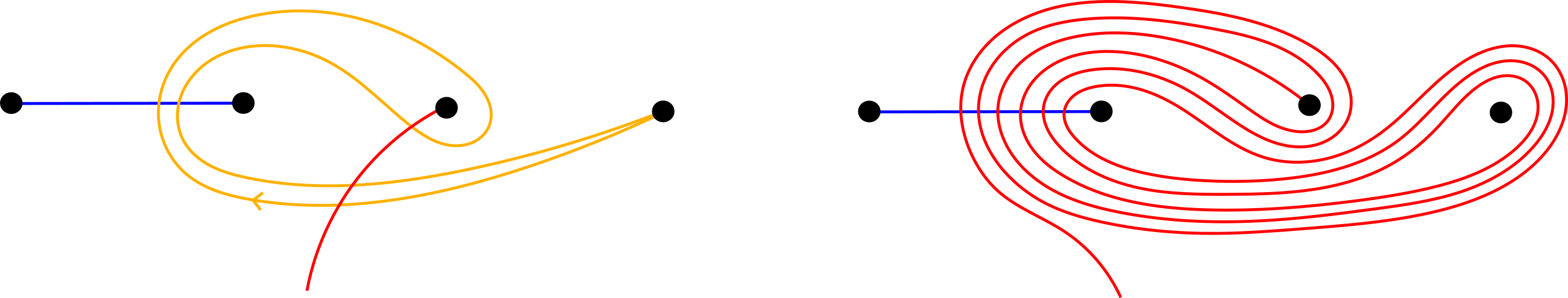}
            \put(5, 13.5){\textcolor{blue}{$\alpha$}}
            \put(17, 1){\textcolor{red}{$\beta_*^3$}}

            \put(58, 13){\textcolor{blue}{$\alpha$}}
            \put(68, 1){\textcolor{red}{$\beta$}}
        \end{overpic}
        \caption{The gold loop represents an element of the subgroup $K_4$. The image $\beta$ of $\beta_*^3$ under the push-map around the gold loop is shown in red on the right. }
        \label{figure:cp-example}
    \end{figure}
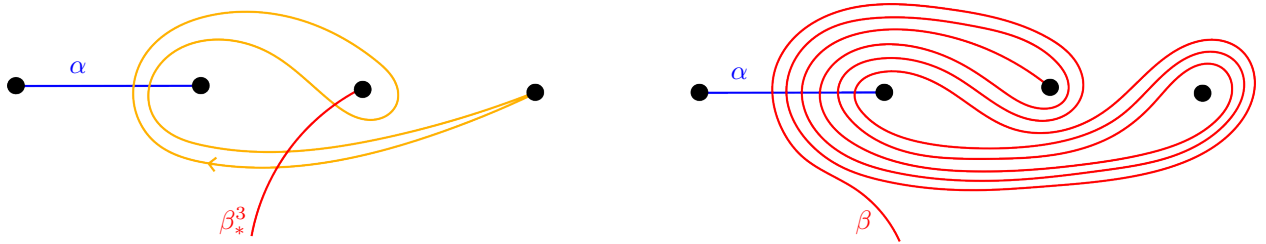
    
  In this part of the paper, we do not need to restrict to the case $n = 4$, and so we will work in the disk $D_n$, the braid group $\B_n$, and its subgroup $K_n$ for $n \geq 4$.  Before we begin, we note the following.
\begin{itemize}
\item Without loss of generality, we can always choose a representative of a loop $\Gamma$ that intersects itself minimally and transversely, and hence there are finitely many points of self-intersection.  
\item The results in this section hold if we replace $K_n$ with any other point-pushing subgroup $K_j$ where $j \geq 4$.
\item In this section we do not assume that our push-maps necessarily correspond to simple loops.  However, we only require simple loops in order to prove Theorem~\ref{theorem:k4-faithful}  below.  
\end{itemize}

\paragraph{Bigon-forming polygons.}  Let $\Gamma \in K_n$, and let $\beta$ be any simple arc joining $p_*$ to $p_3$.  We wish to identify local configurations of subarcs of $\alpha, \beta$, and $\Gamma$ that lead to the creation of a bigon between $(\beta) \Gamma$ and $\alpha$.  As a warm-up, suppose that as we travel along $\Gamma$ there is a subarc of $\beta$ that precedes a bigon between $\Gamma$ and $\alpha$. Then pushing along $\Gamma$ will create a bigon (in fact, two); see Figure~\ref{figure:y-alpha bigon}.  
 \begin{figure}[htpb!]
        \centering
        \begin{subfigure}{.45\linewidth}\centering
                    \includegraphics[width=0.9\linewidth]{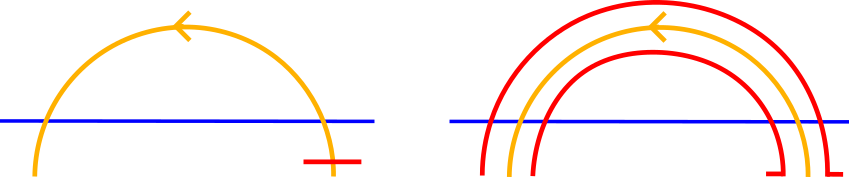}
        \caption{A bigon between $\Gamma$ (gold) and $\alpha$ (blue).  }
        \label{figure:y-alpha bigon}
        \end{subfigure}
        \begin{subfigure}{.45\linewidth}\centering
        \begin{overpic}[width=0.9\linewidth]{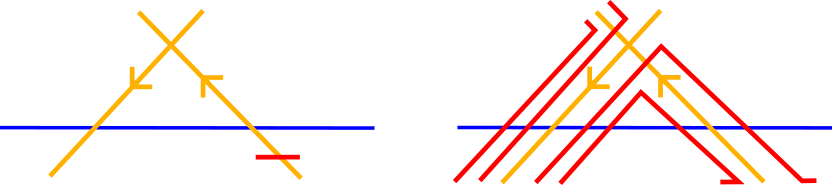}
            \put(40, 2){\textcolor{blue}{$\alpha$}}
            \put(6, 10){\textcolor[RGB]{255,177,0}{$\Gamma^{(2)}$}}
            \put(30, 10){\textcolor[RGB]{255,177,0}{$\Gamma^{(1)}$}}
        \end{overpic}
        \caption{A generalized bigon of length 2 between $\Gamma$ and $\alpha$.}
                \label{figure:generalized bigon}
        \end{subfigure}
        \label{figure:bigons and generalized bigons}
        \caption{}
    \end{figure}
Of course, in general we can avoid such a scenario simply by assuming that all curves are in pairwise minimal position. However, this points us to a more general kind of phenomenon in the case where $\Gamma$ is not simple: a sequence of subarcs of $\Gamma$ that, taken together, effectively form a bigon with $\Gamma$, as shown in Figure~\ref{figure:generalized bigon}.  Then, depending on the ordering of these subarcs (as determined by the orientation of $\Gamma$), pushing along $\Gamma$ can similarly create a bigon. 

To make this more precise, given two subarcs $\Gamma^{(1)}, \Gamma^{(2)}$ of $\Gamma$ in $D_n$ with disjoint interiors, we will say that $\Gamma^{(1)}$ {\it precedes} $\Gamma^{(2)}$ and write $\Gamma^{(1)}< \Gamma^{(2)}$ if the initial point of $\Gamma^{(1)}$ precedes the initial point of $\Gamma^{(2)}$ as we trace out $\Gamma$.  We will further refer to a union of subarcs $\Gamma^{(1)}, \ldots,  \Gamma^{(\ell)}$ that are pairwise disjoint as a {\it piecewise $\Gamma$-arc of length $\ell$} if the following conditions are satisfied:
\begin{itemize}
    \item for each $i \in \{1, \ldots, \ell - 1 \}$, the final point of $\Gamma^{(i)}$ is a self-intersection point of $\Gamma$ in $D_n$, such that the final point of $\Gamma^{(i)}$ is the initial point of $\Gamma^{(i+1)}$;
    \item $\Gamma^{(1)} < \Gamma^{(2)} < \cdots < \Gamma^{(\ell)}$; and 
    \item the union $\Gamma^{(1)} \cup \cdots \cup \Gamma^{(\ell)}$ in $D_n$ contains no loops.  
\end{itemize}

We will now identify three types of local configurations (in addition to bigons) between $\alpha$, $\beta$, and $\Gamma$ that lead to the creation of bigons between $\alpha$ and $(\beta) \Gamma$.  
\begin{enumerate}
    \item{{\bf Generalized bigons.}}  
    Consider now a piecewise $\Gamma$-arc $\gamma$ of length $\ell$ where the initial point of $\Gamma^{(1)}$ and the final point of $\Gamma^{(\ell)}$ are both points of intersection of $\Gamma$ with $\alpha$, so that $\gamma$ together with the $\alpha$-subarc with these two endpoints forms a  $(\ell+1)$-gon in $D_n$ containing no marked points in its interior.  We will refer to such an $(\ell+1)$-gon as a {\it generalized bigon}.   Figure~\ref{figure:generalized bigon} shows a generalized bigon with $\ell = 2$.   If the piecewise $\Gamma$-arc in a generalized bigon is preceded by a point of intersection with $\beta$, then $(\beta) \Gamma$ and $\alpha$ will form a bigon in the same way as if the piecewise $\Gamma$-arc were replaced by a single subarc of $\Gamma$.  We emphasize that the orientation and ordering of the $\Gamma$-subarcs as we traverse $\Gamma$ is crucial here, since otherwise $\Gamma$ would not push the $\beta$-subarc all the way around to form a bigon between $(\beta) \Gamma$ and $\alpha$.  Finally, a ``standard'' bigon is a generalized bigon with a piecewise $\Gamma$-arc $\gamma$ of length 1.
    
    \item{{\bf Generalized trigons.}}  
    Figure~\ref{figure:b-to-a trigon} shows another configuration of curves that leads to the creation of a bigon, namely a trigon in $D_4$ whose boundary consists of a subarc from each of the arcs $\alpha$ and $\beta,$ and the loop $\Gamma$, where the $\Gamma$-subarc is oriented from the $\beta$-arc to the $\alpha$-arc.  We will refer to such a configuration of curves as a {\it $\beta$-to-$\alpha$ trigon}.  The orientation of the $\Gamma$-subarc is crucial here: if its orientation were reversed, then $(\beta) \Gamma$ does not form any bigons with $\alpha$ in this region.  
As with bigons, we can replace the single subarc of $\Gamma$ here with a piecewise $\Gamma$-arc of any length; we refer to any such configuration as a {\it generalized $\beta$-to-$\alpha$ trigon}.  See Figure~\ref{figure:generalized trigon} for an example.
     \begin{figure}[htpb!]
        \centering
        \begin{subfigure}{.45\linewidth}\centering
        \begin{overpic}[width=0.7\linewidth]{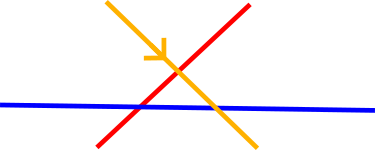}
            \put(65, 30){\textcolor{red}{$\beta$}}
            \put(25, 30){\textcolor[RGB]{255,177,0}{$\Gamma$}}
            \put(5, 16){\textcolor{blue}{$\alpha$}}
        \end{overpic}
        \caption{A $\beta$-to-$\alpha$ trigon.}
        \label{figure:b-to-a trigon}
        \end{subfigure}
        \begin{subfigure}{.5\linewidth}\centering
                    \begin{overpic}[width=0.87\linewidth]{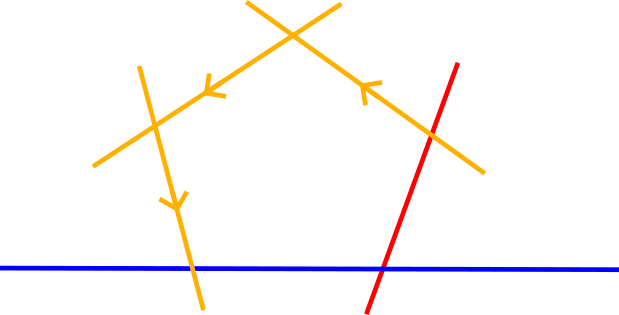}
                        \put(67, 14){\textcolor{red}{$\beta$}}
                        \put(55, 42){\textcolor[RGB]{255,177,0}{$\Gamma^{(1)}$}}
                        \put(30, 41){\textcolor[RGB]{255,177,0}{$\Gamma^{(2)}$}}
                        \put(18, 14){\textcolor[RGB]{255,177,0}{$\Gamma^{(3)}$}}
                        \put(4, 10){\textcolor{blue}{$\alpha$}}
                    \end{overpic}
                \caption{A five-sided example of a generalized $\beta$-to-$\alpha$ trigon.}
                \label{figure:generalized trigon}
        \end{subfigure}
        \caption{}
    \end{figure}

\item{{\bf Generalized rectangles.}} 
Another configuration of curves that may create bigons between $\alpha$ and $(\beta) \Gamma$ arises when two parallel subarcs of $\Gamma$ traveling between $\beta$ and $\alpha$ form a rectangle containing no marked points; we will call this a $\beta$-to-$\alpha$ rectangle.  As with bigons and trigons, we can use two  piecewise $\Gamma$-arcs instead of two $\Gamma$-arcs, and we refer to any such configuration as a {\it generalized rectangle}; see Figure~\ref{figure:rectangle} for an example.  We emphasize that in the case of a generalized rectangle, we do not necessarily require the two piecewise $\Gamma$-arcs to be disjoint.  In other words, we allow ``degenerate'' generalized rectangles, in which some number of the $\Gamma$-segments involved can play a role in both piecewise $\Gamma$-arcs; we will see examples of this in the proof of the next lemma (see Figure~\ref{figure:One point case 2 region 2}).  Note also that every generalized rectangle includes two crossings of $\Gamma$ and $\alpha$ with the same sign that are adjacent, as the crossings are ordered along $\alpha$. 

\begin{figure}[htpb!]
    \centering
     \begin{subfigure}{.35\linewidth}\centering
    \begin{overpic}[width=0.7\linewidth]{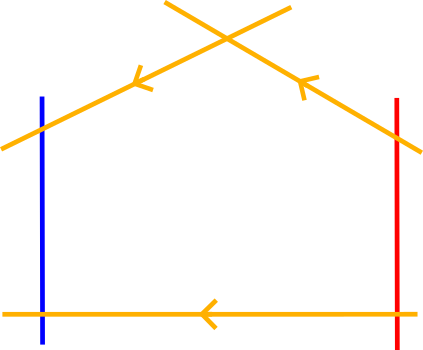}
        \put(38,13){\textcolor[RGB]{255,177,0}{$\Gamma$}}
        \put(2, 28){\textcolor{blue}{$\alpha$}}
        \put(98,28){\textcolor{red}{$\beta$}}
    \end{overpic}
    \caption{An example of a generalized rectangle with piecewise $\Gamma$-arcs of length 1 and length 2.}
    \label{figure:rectangle}
    \end{subfigure}
    \hskip.1in
     \begin{subfigure}{.45\linewidth}\centering
        \begin{overpic}[width=0.7\linewidth]{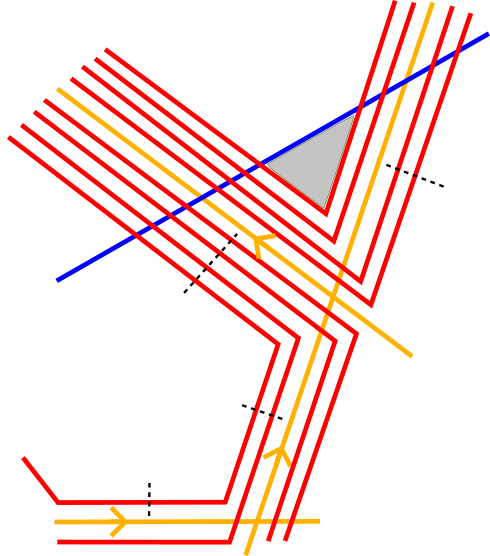}
            \put(11, 12){\textcolor{red}{$\beta$}}
            \put(13, 45){\textcolor{blue}{$\alpha$}}
            \put(25, 15){$\Gamma^{(1)}$}
            \put(32, 26){$\Gamma^{(2)}$}
            \put(26, 40){$\Gamma^{(3)}$}
            \put(82, 64){$\Gamma^{(4)}$}
        \end{overpic}
        \caption{A degenerate generalized rectangle formed by \newline $\mbox{\qquad} \Gamma^{(1)}\cup\Gamma^{(2)}\cup\Gamma^{(3)}$ and $\Gamma^{(1)}\cup\Gamma^{(2)}\cup\Gamma^{(4)}$.}  
        \label{figure:One point case 2 region 2}
        \end{subfigure}
        \caption{}
\end{figure}

\item {\bf Generalized bigons/trigons/rectangles with basepoint.}  Let $\Gamma_I$ denote the subarc of $\Gamma$ from $p_4$ to its first point of intersection with $\alpha$, and similarly let $\Gamma_F$ denote the subarc of $\Gamma$ from its final point of intersection with $\alpha$ back to $p_4$.  We next consider a configuration that would become a generalized polygon of any of the previous three types upon the deletion of the puncture $p_4$.   If such a configuration also has one additional subarc $\gamma$ of $\Gamma$ from a $\beta$-subarc of the generalized polygon to $\alpha$, satisfying the following three properties:
\begin{enumerate}
    \item $\gamma$ is the union of a subarc of $\Gamma_I$ together with a subarc of $\Gamma_F$ (including the basepoint $p_4$);
    \item $\gamma$ does not cross $\alpha$ or $\beta$ except at its endpoints;
    \item$\Gamma_I$ crosses $\alpha$; and
    \item $\Gamma_F$ crosses $\beta$.
\end{enumerate}
Such a configuration also results in a bigon after applying $\Gamma$ to the arc $\beta$; see Figure~\ref{figure:additional generalized rectangle} for an example.

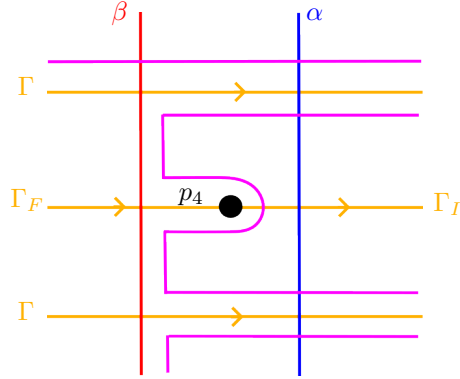
\begin{figure}[htpb!]
    \centering
\begin{overpic}[width=0.3\linewidth]{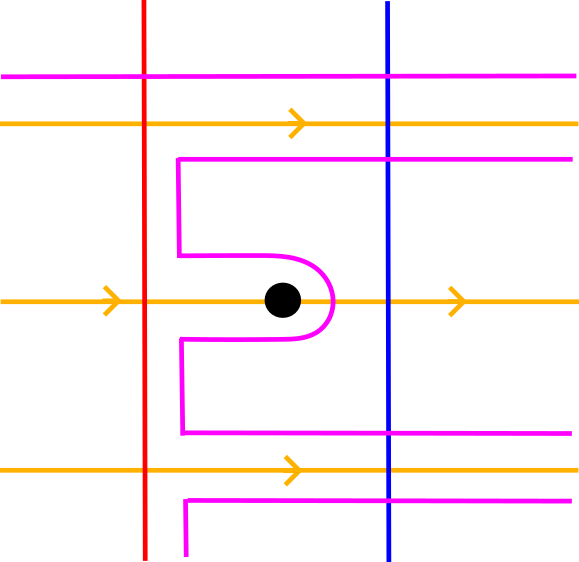}
            \put(17, 95){\textcolor{red}{$\beta$}}
            \put(69, 95){\textcolor{blue}{$\alpha$}}
            \put(-8, 75){\textcolor[RGB]{255,177,0}{$\Gamma$}}
            \put(-10, 45){\textcolor[RGB]{255,177,0}{$\Gamma_F$}}
            \put(103, 44){\textcolor[RGB]{255,177,0}{$\Gamma_I$}}
             \put(35, 47){$p_4$}
             \put(-8, 15){\textcolor[RGB]{255,177,0}{$\Gamma$}}
        \end{overpic}
        \caption{A generalized rectangle with basepoint.}
    \label{figure:additional generalized rectangle}
\end{figure}
\end{enumerate}

In the proof of the next lemma, we will see that generalized bigons, $\beta$-to-$\alpha$ trigons, and rectangles, possibly with basepoint, are the only types of configurations that give rise to bigons between $(\beta) \Gamma$ and $\alpha$.  Hence we will refer to these configurations collectively as {\it bigon-forming polygons}.  

\p{Proper products.}  Let $\Phi \in K_n$, and let $\beta =  (\beta_*^3) \Phi$. 
Our next step will be to establish a set of conditions on a loop $\Gamma$ based at $p_n$ that will imply $(\beta) \Gamma$ and $\alpha$ are in minimal position.  To that end, we let $\Gamma_I$ denote the {\it initial component} of $\Gamma$, that is, the subarc of $\Gamma$ from $p_n$ to the first crossing of $\Gamma$ with $\alpha$.  Similarly, we let $\Gamma_F$ denote the {\it final component} of $\Gamma$, traveling from its final crossing with $\alpha$ back to $p_n$. We also let $\beta'$ denote the $\beta$-component of the innermost disk in the disk sequence of $\Phi$ that contains only $p_n$ if such a disk exists; otherwise we set $\beta' = \emptyset$.  We say that $\Phi \cdot \Gamma$ is a {\it proper product} if we can choose a representative of the loop $\Gamma$ that is in minimal position with respect to $\alpha$ and $\beta$, and satisfying the following two conditions: 
\begin{enumerate}
    \item there are no bigon-forming polygons between $\alpha$, $\beta$, and $\Gamma$; and 
    \item one of the following two conditions holds:
\begin{enumerate}
    \item  $\Gamma_I \cap \beta' \neq \emptyset$; or
    \item  $\Gamma_I \cap \beta' = \emptyset$ and $\Gamma_F \cap \beta' = \emptyset$. 
\end{enumerate}
\end{enumerate}
 The second condition ensures that postcomposing with $\Gamma$ does not ``undo'' the creation (by $\Phi$) of a disk containing only $p_4$. Figure~\ref{figure:proper products} depicts the local picture for each type of configuration: in first case (on the left in Figure~\ref{figure:proper products}), reversing the orientation of $\Gamma$ has the effect of swapping the roles of $\Gamma_I$ and $\Gamma_F$, resulting in a loop that does not form a proper product.  In practice, we will generally be given a push-map $\Phi$, and then seek the second factor $\Gamma$ to form a proper product with $\Phi$.  We will also refer to $\Phi \cdot \Gamma_1 \cdot \Gamma_2$ as a proper product if both $\Phi \cdot \Gamma_1$ and $(\Phi \cdot \Gamma_1) \cdot \Gamma_2$ are proper products.

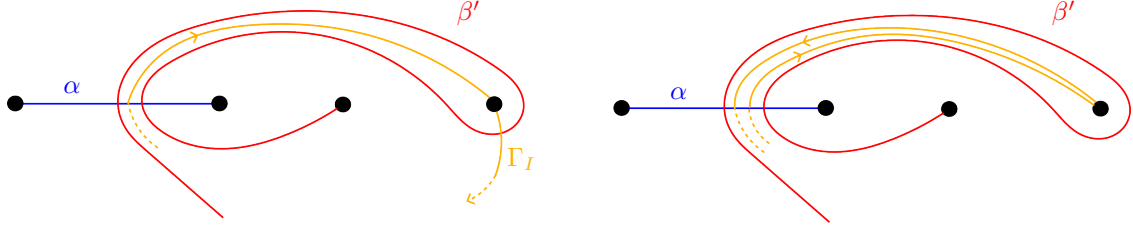
\begin{figure}[htpb!]
    \centering
    \begin{overpic}[width=0.9\linewidth]{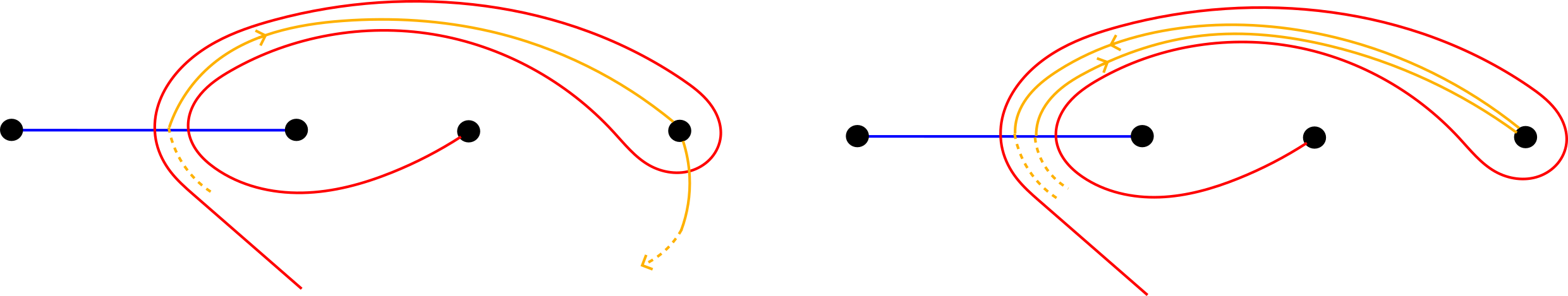}
      \put(40, 18){\textcolor{red}{$\beta'$}}
       \put(93, 18){\textcolor{red}{$\beta'$}}
            \put(44.5, 5){\textcolor[RGB]{255,177,0}{$\Gamma_I$}}
            \put(5, 11.5){\textcolor{blue}{$\alpha$}}
             \put(59, 11){\textcolor{blue}{$\alpha$}}
    \end{overpic}
    \caption{Two examples of the local picture for Condition 2 of a proper product $\Phi \cdot \Gamma$ indicating the initial and final components of $\Gamma$ in each case; reversing the orientation of $\Gamma$ in the left-hand figure gives a non-example.}
    \label{figure:proper products}
\end{figure}

The next lemma gives a straightforward criterion for ensuring that the product of two push-maps does not yield any bigons between $\alpha$ and the image of $\beta_*^3$ under the product.

\begin{proposition}
    \label{prop:ypush_minimalposition}
    Let $\Phi \in K_n$, and let $\beta = \Phi (\beta_*^3)$. If $\Gamma$ is a loop based at $p_n$ such that $\Phi \cdot \Gamma$ is a proper product, then the arcs $(\beta) \Gamma$ and $\alpha$ are in minimal position. 
\end{proposition}

\begin{proof}
We begin by choosing a specific representative of $\Gamma$ so that $\alpha, \beta$ and $\Gamma$ are pairwise in minimal position with no triple points of intersection.  Up to isotopy we can assume that $(\beta)\Gamma$ and $\alpha$ are transverse in $D_n$.  Consider now a point $s \in (\beta)\Gamma \cap \alpha$.  If $s$ also lies in $\beta \cap \alpha$ we will refer to this as an {\it original} point of intersection; otherwise we say that $s$ is a {\it new} point of intersection.   For the sake of simplicity, in what follows we will refer only generalized bigons, trigons, and rectangles; the proof goes through when adding in each case with basepoint.

Suppose now that $\alpha$ and $(\beta) \Gamma$ form a bigon, and suppose further that the two vertices of this bigon are both new points of intersection.  New points of intersection occur when a $\beta$-subarc precedes an $\alpha$-subarc as we travel along $\Gamma$. Therefore our two vertices correspond to adjacent crossings of $\Gamma$ and $\alpha$ of the same sign or of different signs, as shown in Figure~\ref{fig:gamma-signs}.
\begin{figure}[htpb!]
    \centering
    \vspace{3mm}
    \begin{overpic}[width=0.6\linewidth]{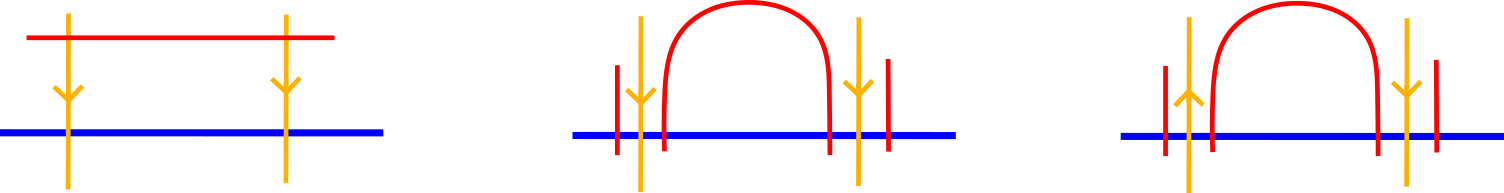}
        \put(48,1){\textcolor{blue}{$\alpha$}}
        \put(47,15){\textcolor{red}{$(\beta) \Gamma$}}
        \put(58,-1){\textcolor[RGB]{255,177,0}{$\Gamma$}}
    \end{overpic}
    \caption{On the left two subarcs of $\Gamma$ intersect $\beta$ with the same sign; on the right two subarcs of $\Gamma$  intersect $\beta$ with different signs.  }
    \label{fig:gamma-signs}
\end{figure}
If they are the same sign, this implies the existence of a generalized rectangle between $\Gamma$, $\alpha$, and $\beta$, which contradicts our assumption that $\Phi \cdot \Gamma$ is a proper product.  If they are different signs, this implies the existence of a generalized bigon between $\Gamma$ and $\alpha$, which again contradicts our assumption that $\Phi \cdot \Gamma$ is a proper product.  Therefore at least one vertex of any bigon between $\alpha$ and $(\beta) \Gamma$ must be an original point of intersection.

Now, the support of the push-map $\Gamma$ is a regular neighborhood of $\Gamma$ in $D_n$.  Hence we can decompose $(\beta)\Gamma$ as the union of subarcs of $\beta$ that are fixed by $\Gamma$, which we refer to as {\it original subarcs}, together with {\it new subarcs} that lie in the support of $\Gamma$.  Each of the new subarcs is one of four types:  (I) joins an original subarc to a new point of intersection,  (II) joins two new points of intersection forming part of the boundary of a disk bounding $p_n$, (III) joins two original subarcs, or (IV) joins two new points of intersection in a disk not containing only $p_n$.

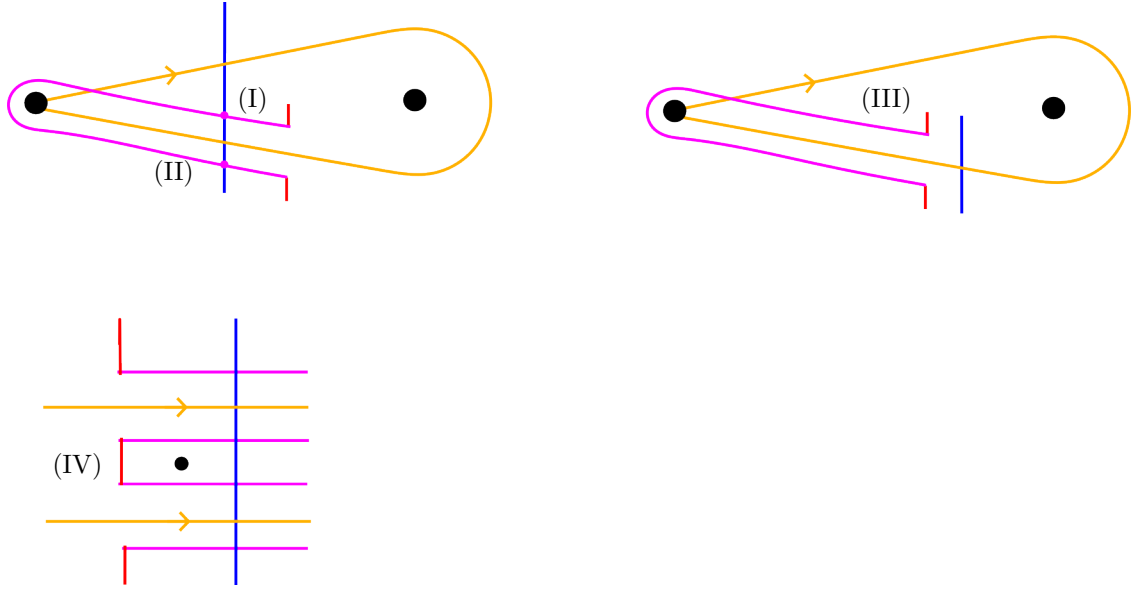
\begin{figure}[htpb!]
    \centering
    \begin{overpic}[width=0.9\linewidth]{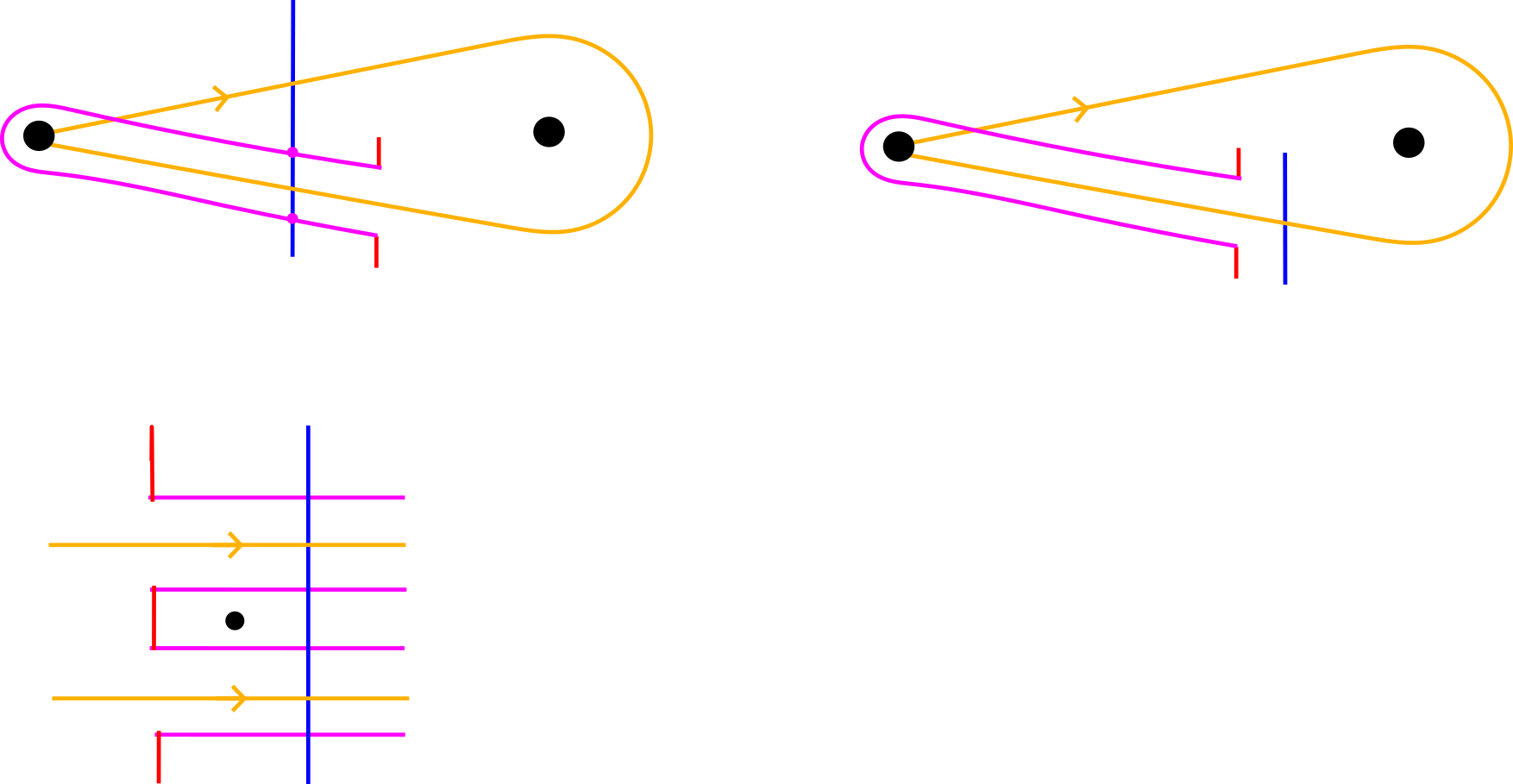}
        \put(13,36){(II)}
        \put(20.5,42.5){(I)}
        \put(76,42.5){(III)}
        \put(4,10){(IV)}
    \end{overpic}
    \caption{Subarcs of types (I)-(IV): here original subarcs of $\beta$ appear in red, while new subarcs of $\beta$ are colored purple.}
    \label{fig:new-crossings}
\end{figure}

Suppose now that the vertices $s_1, s_2$ of our bigon are both original points of intersection.  Then the subarc of $(\beta) \Gamma$ that joins $s_1$ to $s_2$ cannot contain any new subarcs of type (I), type (II), or type (III) and must therefore contain a new subarc of type (IV).  It follows that we have two subarcs $\beta' \subset \beta$ and $\alpha' \subset \alpha$ satisfying the following three statements:
\begin{itemize}
\item the union $\alpha' \cup \beta'$ does not form a bigon;
\item the union $\alpha' \cup (\beta')\Gamma$ forms a bigon; and
\item the 
intersection $\alpha' \cap (\beta')\Gamma = \alpha' \cap \beta' = \{ s_1, s_2 \}$.
\end{itemize}
The disk bounded by $\alpha' \cup \beta'$ in $D_n$ must contain at least one marked point, since we are assuming that $\alpha$ and $\beta$ are in minimal position. If the disk bounded by $\alpha' \cup \beta'$ contained $p_i$ for some $i \neq n$, then the disk bounded by $\alpha' \cup (\beta')\Gamma$ would also contain $p_i$, but this is not possible since $\alpha' \cup (\beta')\Gamma$ bounds a bigon. Therefore the disk bounded by $\alpha' \cup \beta'$ contains only the marked point $p_n$.  We must also have that the final component $\Gamma_F$ of $\Gamma$ intersects $\beta'$ in a single point, since $\alpha' \cup (\beta')\Gamma$ does not contain $p_n$.  Moreover, no other subarc of $\Gamma$ can intersect $\beta'$, as otherwise $\alpha' \cup (\beta')\Gamma$ would not form a bigon.  In particular, it follows that $\Gamma_I$ intersects $\alpha' \cup (\beta')\Gamma$ in a single point in $\alpha'$.  Hence the loop $\Gamma$ must intersect the disk bounded by $\alpha' \cup \beta'$ as shown in Figure~\ref{figure:p4-bigon}.  However, such a configuration contradicts the fact that $\Phi \cdot \Gamma$ is a proper product.  Therefore the vertices $s_1$ and $s_2$ cannot both be original points of intersection.  
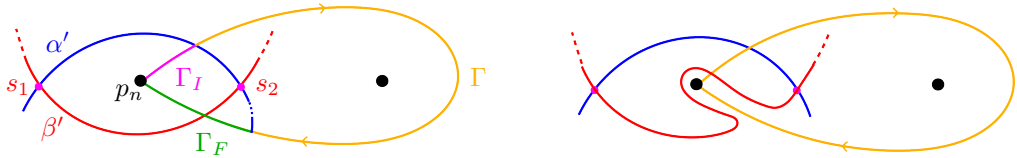
\begin{figure}[htpb!] 
    \centering
    \begin{overpic}[width=.8\linewidth]{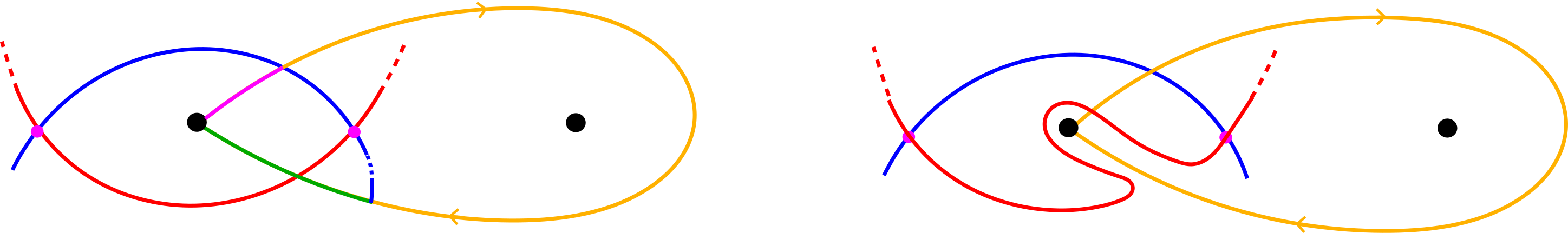}
        \put(3, 10){\textcolor{blue}{$\alpha'$}}
        \put(2.5, 1.5){\textcolor{red}{$\beta'$}}
          \put(-1, 6){\textcolor{red}{$s_1$}}
          \put(24, 6){\textcolor{red}{$s_2$}}
        \put(45.5, 6.5){\textcolor[RGB]{255,177,0}{$\Gamma$}}
        \put(18, 0){\textcolor[RGB]{0,170,0}{$\Gamma_F$}}
        \put(16, 6.5){\textcolor[RGB]{255,0,255}{$\Gamma_I$}}
        \put(10, 5.5){{$p_n$}}
    \end{overpic}
    \caption{A new bigon involving two original points of intersection $s_1$ and $s_2$ between $\alpha$ and $\beta$ must arise from a configuration of curves as depicted. } 
    \label{figure:p4-bigon}
\end{figure}

It remains to consider the case in which precisely one vertex of our bigon between $\alpha$ and $(\beta) \Gamma$ is an original point of intersection; we denote this vertex by $s_1$.  The other vertex $s_2$ must therefore be a new point of intersection formed when $\Gamma$ meets a subarc of $\beta$ and pushes it along until it crosses $\alpha$.  Without loss of generality, we can assume that $s_2$ is the new point of intersection to the left of the loop $\Gamma$, as shown in Figure~\ref{figure:One point case 1}.  
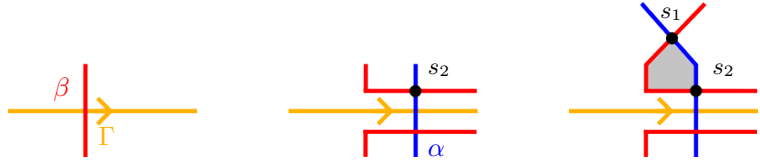
\begin{figure}[htpb!]
        \centering
        \begin{overpic}[width=0.6\linewidth]{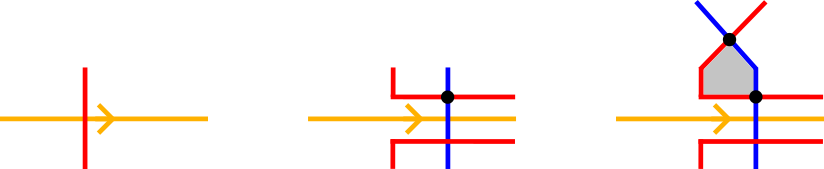}
            \put(12, 1.5){\textcolor[RGB]{255,177,0}{$\Gamma$}}
            \put(6, 8){\textcolor{red}{$\beta$}}
            \put(56, 0){\textcolor{blue}{$\alpha$}}
            \put(56, 11){\small{$s_2$}}
            \put(94, 11){\small{$s_2$}}
            \put(87, 18.5){\small{$s_1$}}
        \end{overpic}
        \caption{The simplest case where a bigon between $\alpha$ and $(\beta) \Gamma$ involves precisely one original point of intersection.}
              \label{figure:One point case 1}
    \end{figure}

Then there are two possible candidate regions for a bigon with $s_2$ as a vertex: one to the left of $\alpha$ as shown in the figure, and one to the right.   In each case, the vertex $s_1$ lies along $\alpha$ on the same side of $\Gamma$ as $s_2$ in the local picture shown in Figure~\ref{figure:One point case 1}. In the first case, there must be a piecewise $\Gamma$-arc forming a generalized $\beta$-to-$\alpha$ trigon with $s_2$ corresponding to its final point of intersection with $\alpha$, which contradicts that $\Phi \cdot \Gamma$ is a proper product; see Figure~\ref{figure:One point case 2 region 1}. 
 \begin{figure}[htpb!] 
    \centering
    \begin{overpic}[width=.7\linewidth]{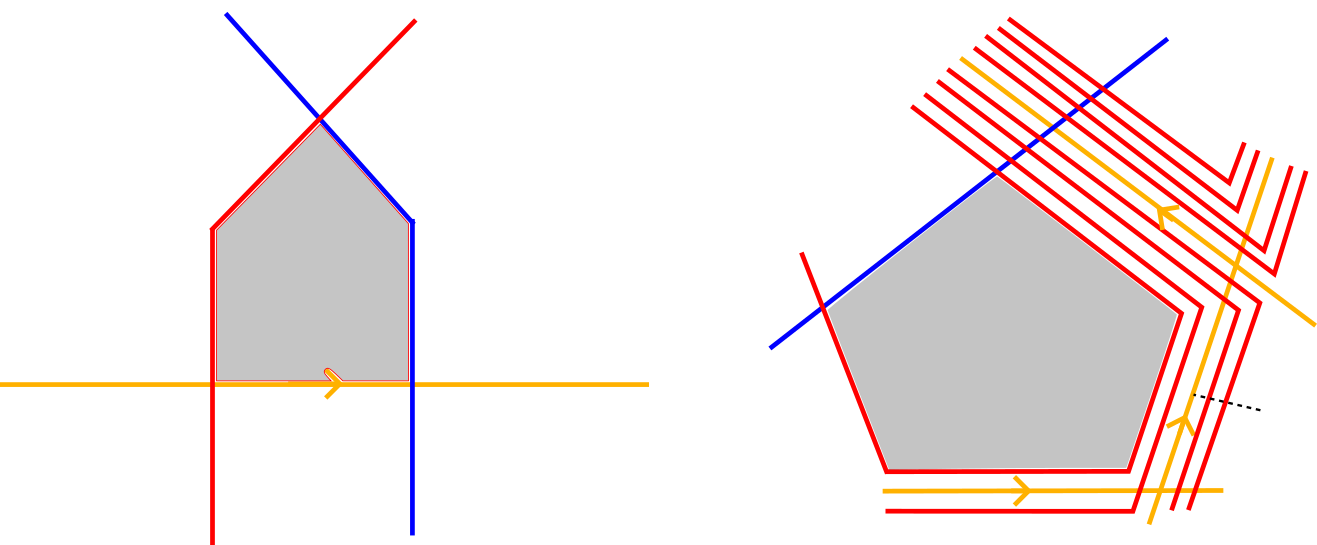}
    \put(23, 35){\textcolor{blue}{$\alpha$}}
    \put(13, 17){\textcolor{red}{$\beta$}}
    \put(97, 9){\textcolor[RGB]{255,177,0}{$\Gamma$}}
    \end{overpic}
    \caption{A possible bigon that arises from a generalized trigon.  The oriented arcs in gold are subarcs of a loop $\Gamma$. }
    \label{figure:One point case 2 region 1}
    \end{figure}
The second case also cannot occur, since this would imply that $(\beta)\Gamma$ continues on along a new subarc from $s_2$ to $s_1$ to form our bigon, which in turn implies that $\Gamma$ itself must form a generalized bigon with $\beta$ and that $s_1$ is also a new point of intersection, both of which are contradictions. The proposition follows. 
\end{proof} 

In the next section, we will restrict our attention to the case $n = 4$.

\section{Disk sequences in the 4-strand braid group}
\label{section:disks in B_4}
The aim of this section is to prove Theorem~\ref{theorem:k4-faithful} below, which states that the Burau representation $\rho_4$ is faithful on its restriction to the point-pushing group $K_4$.  By Proposition~\ref{prop:reduction to K_4}, this will imply our Main Theorem.  Our strategy for proving Theorem~\ref{theorem:k4-faithful} is to use proper products to identify certain braids that satisfy the parity condition.  This will allow us to apply Moody's theorem (Theorem~\ref{theorem:moody2}) to establish faithfulness.  

\subsection{Disks and parity in ${\bm \B_4}$}
We begin with a careful analysis of the types of disks that can arise in the disk sequence of a 4-strand braid.  To that end, we introduce some useful notation: for a disk $\Delta$ contained in $\D_4$, we let $\mcP(\Delta)$ denote the set of marked points that are contained in the interior of $\Delta$, and hence $\mathcal{P}(\Delta) \subseteq \{ p_1, p_2, p_3, p_4 \}$.  We note that up to this point in the paper we have not used the assumption that $n=4$. We will use that assumption in this section, particularly in Lemmas~\ref{lemma:characterizing sign-changing} and \ref{lemma:only 4}. 

\begin{lemma}
\label{lemma:characterizing sign-changing}
    Let $\Delta$ be a disk arising in the disk sequence of an element of $\B_4$.  Then $\Delta$ is sign-preserving if and only if $\mcP(\Delta)$ is one of the following:
    \[ \{p_1, p_3\}, \hspace{.1in} \{p_1, p_4\}, \hspace{.1in}\{p_2, p_3\}, \hspace{.1in}\{p_2, p_4\}, \hspace{.1in}\{p_1, p_3, p_4 \}, \hspace{.1in} \{p_2, p_3, p_4\}. \]
    Moreover, $\Delta$ is sign-changing if and only if $\mcP(\Delta)$ is one of the following:
  \[ \{p_3\}, \hspace{.1in}\{p_4\},  \hspace{.1in}\{p_3, p_4\}, \hspace{.1in}\{p_1, p_2, p_3\}, \hspace{.1in}\{p_1, p_2, p_4\},  \hspace{.1in}\{p_1, p_2, p_3, p_4\}. \]
\end{lemma}

\begin{proof}
Let $\Delta$ be a disk in the disk sequence of an element $\Phi \in \B_4$, and let $q$ and $q'$ denote the two points of intersection between $\alpha$ and $\beta = (\beta_*^3) \Phi$ that lie in $\partial \Delta$.  Without loss of generality, we assume that $q$ lies between $p_1$ and $q'$ on $\alpha$, and that $q$ precedes $q'$ as we traveling along $\beta$.  

Suppose that $\mcP(\Delta) = \{p_1\}$.  Since $\beta$ and $\alpha$ intersect transversely at $q$, the subarc of $\beta$ joining the basepoint $p_*$ to $q$ must intersect $\partial \Delta$.  Hence this subarc either intersects the $\beta$-component of $\partial \Delta$ or the subarc of $\alpha$ with endpoints $q$ and $q'$.  The former contradicts our assumption that $\beta$ is simple, and the latter contradicts our assumption that $\beta$ and $\alpha$ are in minimal position.  Therefore $\mcP(\Delta) \neq \{p_1\}$.  By a similar argument, we see that $\mcP(\Delta) \neq \{p_2\}$ and that $\mcP(\Delta) \neq \{p_1, p_2\}$. 

We next claim that $\Delta$ is sign-preserving if and only if it contains precisely one of the two endpoints of $\alpha$; the lemma follows immediately from the claim. Suppose first that $\Delta$ is sign-preserving.  Without loss of generality, we assume that $\beta$ is oriented down (in our standard picture where $\alpha$ is horizontal) at both $q$ and $q'$, and that $\Delta$ contains $p_1$; see Figure \ref{fig:sign-preserving}.

\begin{figure}[htpb!]
    \centering
    \begin{overpic}[width=0.3\linewidth]{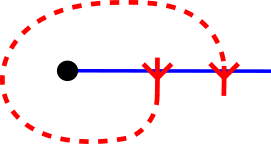}
        \put(56,35){$q$}
        \put(81,12){$q'$}
    \end{overpic}
    \caption{A sign-preserving disk}
    \label{fig:sign-preserving}
\end{figure}

Then $\partial \Delta$ is the union of a subarc of $\beta$ and a subarc of $\alpha$, both with endpoints $q, q'$, and the interior of $\Delta$ lies to the right of $\beta$ as we travel along $\beta$ from $q$ to $q'$.  This implies that $\Delta$ does not contain the subarc of $\alpha$ with endpoints $q'$ and $p_2$; in particular $\Delta$ does not contain $p_2$.  A similar argument shows that if $\Delta$ is sign-changing and contains $p_1$, then it also contains $p_2$.  The claim follows, and we are done.  
\end{proof}

\subsection{The point-pushing subgroup ${\bm K_4}$} 

Our focus in the remainder of this section will be the point-pushing subgroup $K_4$.  Our next step will be to analyze the types of disks that can appear in the disk sequence of certain braids in $K_4$.  Specifically, we will be constructing proper products of certain elements in $K_4$ to which we can apply Lemma~\ref{lemma:no-cancellations}.  We will then analyze the corresponding disk sequences, which will ultimately provide a mechanism for applying Theorem~\ref{theorem:moody2}.  

\begin{lemma}
\label{lemma:piecewise disjoint}
    Let $\Phi \in K_4$, let $\Gamma \in K_4$ be a push-map along a simple loop, and suppose that $\Phi = \Psi \cdot \Gamma$ is a proper product.  Let $\Delta_1,\dots,\Delta_{m-1}$ denote the disk sequence of $\Phi$. Then, for any fixed $i$, either $\Delta_i$ is a 1-disk containing only $p_4$, or there is an isotopy of $\beta = (\beta_*^3)\Phi$ so that the $\beta$-component of the boundary $\Delta_i$ is disjoint from $\Gamma$.
\end{lemma}

Before giving the proof, we emphasize that Lemma~\ref{lemma:piecewise disjoint} only guarantees that the $\beta$-component of any single disk $\Delta_i$ can be isotoped off $\Gamma$.  It does not guarantee that each such $\beta$-component can be isotoped off $\Gamma$ simultaneously; indeed, this is not possible in general.  We also introduce some notation that will be convenient in the proof of the lemma and in what follows: if $\gamma$ is an arc or a loop in the disk $D_n$, we will refer to a disk whose boundary is the union of a single subarc of $\alpha$ together with a single subarc of $\gamma$ as an $\alpha$-$\gamma$ disk. 

\begin{proof}
    Let $\beta' = (\beta_*^3)\Psi$. Again, we choose fixed representatives of $\Gamma$ and $\alpha$ that are in minimal position.  By Proposition~\ref{prop:ypush_minimalposition}, we may choose a representative of $\beta'$ so that $(\beta')\Gamma $ is in minimal position with respect to $\alpha$.  Since $\Gamma$ is a simple loop , we can represent the action of the push-map $\Gamma$ via the schematic in Figure~\ref{figure:local-pictures}. 
    \begin{figure}[htpb!]
        \centering
        \vspace{2mm}
        \begin{subfigure}{.45\linewidth}
            \begin{overpic}[width=0.9\linewidth]{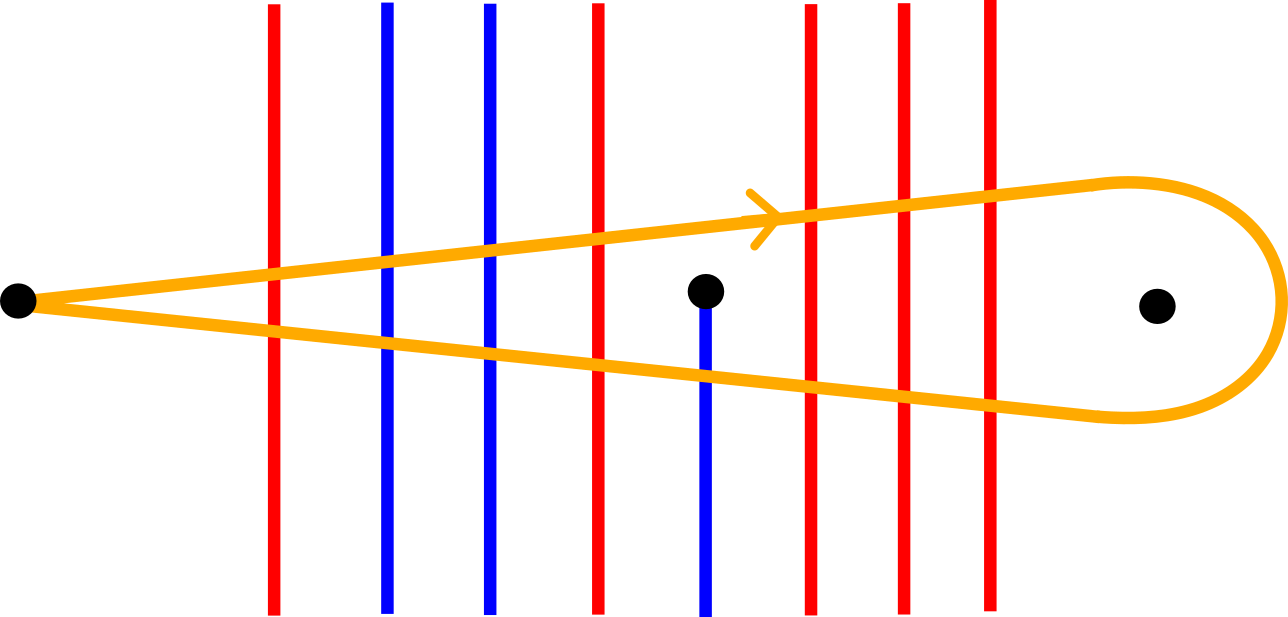}
                \put(32, 50){\textcolor{blue}{$\alpha$}}
                \put(68, 50){\textcolor{red}{$\beta$}}
                \put(85, 37){\textcolor[RGB]{255,177,0}{$\Gamma_1$}}
            \end{overpic}
        \caption{The local picture of pushing along $\Gamma_1$.}
        \label{figure:local-picture1}
        \end{subfigure}
        \begin{subfigure}{.45\linewidth}
        \includegraphics[width=0.9\linewidth]{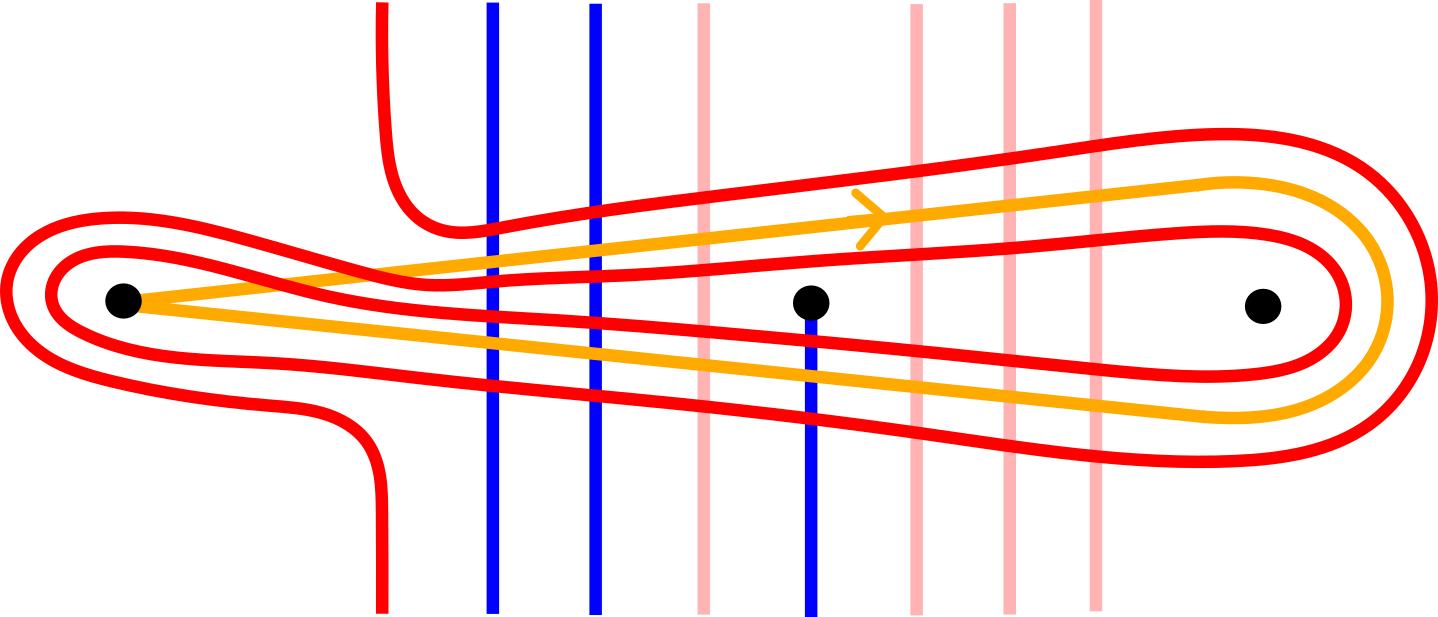}
        \caption{Pushing the first $\beta$ strand}
        \label{figure:local-picture2}
        \end{subfigure}
        \caption{Here blue arcs represent subarcs of $\alpha$, while red/pink arcs represent subarcs of $\beta$.}
        \label{figure:local-pictures}
    \end{figure}
    Figure~\ref{figure:local-picture2} shows the resulting local picture for $\beta = (\beta_*^3)\Phi$ after pushing $\beta'$ along $\Gamma$, possibly creating a number of new disks in the associated disk sequence.  For each such disk that does not contain $p_4$, Figure~\ref{figure:local-picture2} shows that its $\beta$-component is disjoint from $\Gamma$.  For the remaining $\alpha$-$\beta$ disks in Figure~\ref{figure:local-picture2}, there is an isotopy of $\beta$ in $D_n$ that preserves the disk sequence while removing the intersection of the $\beta$-component of these disks from $\Gamma$. This isotopy can be visualized by sliding the ``topmost'' part of any $\beta$-component of the boundary of a disk that intersects $\Gamma$ in Figure~\ref{figure:local-picture2} upwards along the rightmost subarc of $\alpha$ that appears in the local picture; one can perform a single isotopy that moves all such $\beta$-subarcs simultaneously off $\Gamma$. 
\end{proof}

\p{A particular push-map.}  Let $\Gamma_1$ be the push-map corresponding to the loop shown in Figure~\ref{figure:gamma1}.
\begin{figure}[htpb!]
    \centering
    \includegraphics[width=0.6\linewidth]{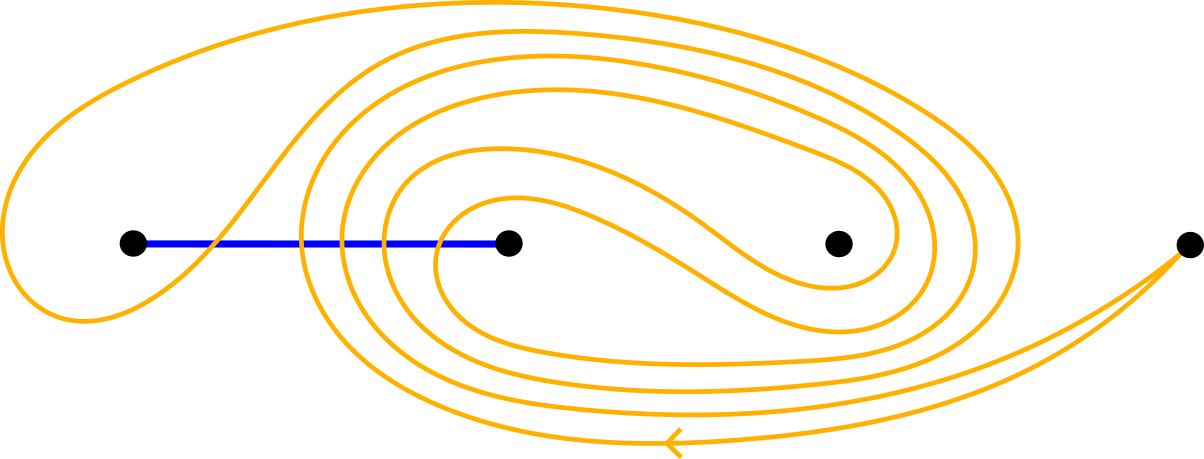}
    \caption{A particular push map $\Gamma_1$. }
    \label{figure:gamma1}
\end{figure}

\begin{lemma}
\label{lemma:only 4}
    Let $\Phi\in K_4$, let $\Gamma_1\in K_4$ be the push-map along the simple loop shown in Figure~\ref{figure:gamma1}, and suppose that $\Phi = \Psi \cdot \Gamma_1$ is a proper product. Let $\Delta$ be a disk arising in the disk sequence of $\Phi$. Then $\Delta$ is sign-preserving if and only if $\mcP(\Delta)$ is the following:
    \[ \hspace{.1in}\{p_2, p_3\}. \]
    Moreover, $\Delta$ is sign-changing if and only if $\mcP(\Delta)$ is one of the following:
  \[ \{p_3\}, \hspace{.1in}\{p_4\}, \hspace{.1in}\{p_1, p_2, p_3\},  \hspace{.1in}\{p_1, p_2, p_3, p_4\}. \]
\end{lemma}
\begin{proof}
    By Lemma \ref{lemma:piecewise disjoint} any disk $\Delta$ in the disk sequence of $\Phi$ can have its $\beta$-component isotoped off of $\Gamma_1$. Referring to Figure~\ref{figure:wn-disks-finite}, we see that, up to equivalence, there are only finitely many possible $\beta$-subarcs forming an $\alpha$-$\beta$ disk that is disjoint from $\Gamma_1$. 
      \begin{figure}[htpb!]
        \centering
        \begin{overpic}[width=0.45\linewidth]{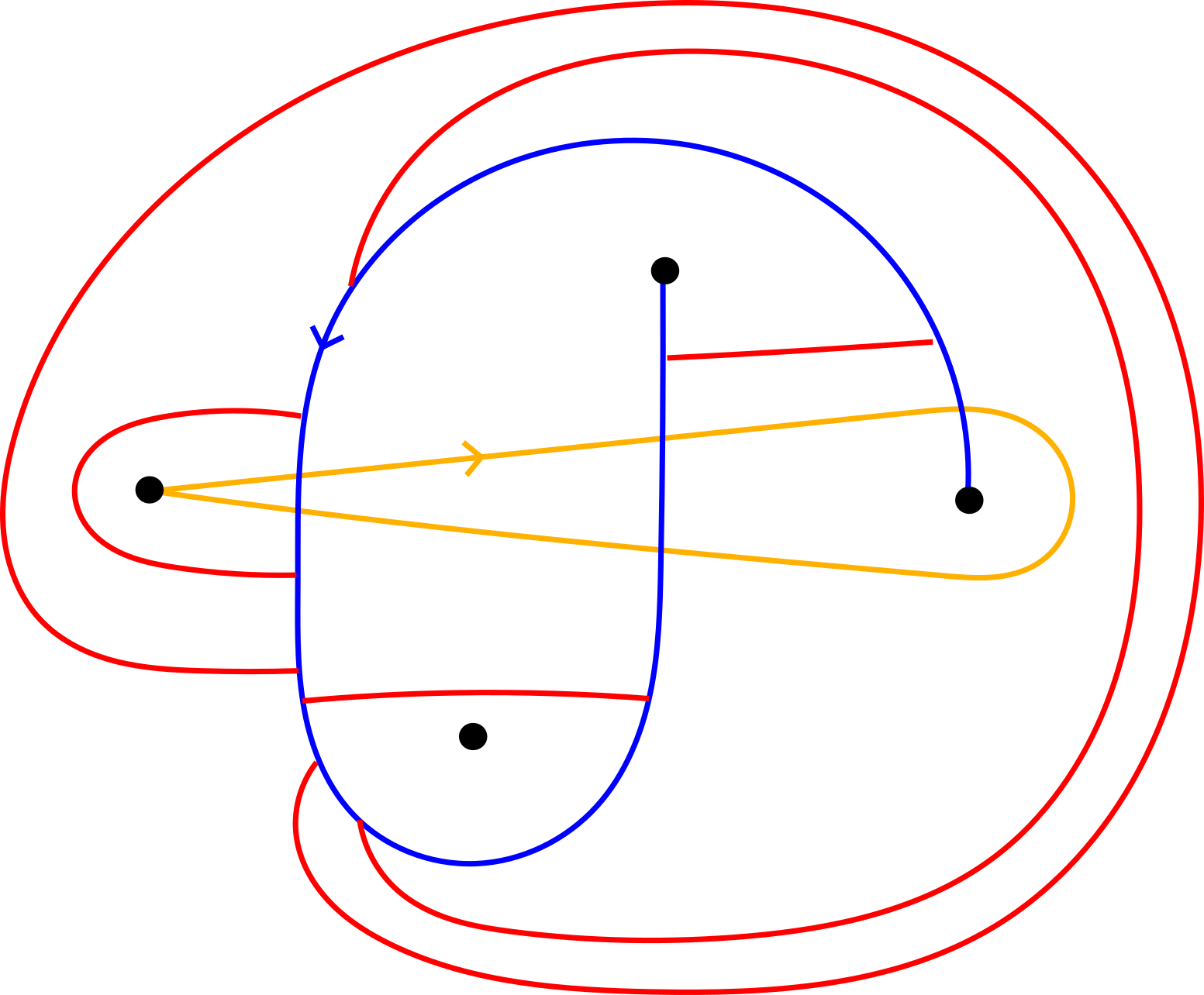}
        \put(10, 38){$p_4$}
        \put(37, 17){$p_3$}
        \put(74, 40){$p_1$}
        \put(53, 64){$p_2$}
        \end{overpic}
        \caption{The loop $\Gamma_1$ is shown in gold, together with red arcs indicating all possible disks arising in the disk sequence of a braid in $K_4$ such that the $\beta$-component of the disk can be isotoped off of $\Gamma_1$.}
        \label{figure:wn-disks-finite}
    \end{figure}
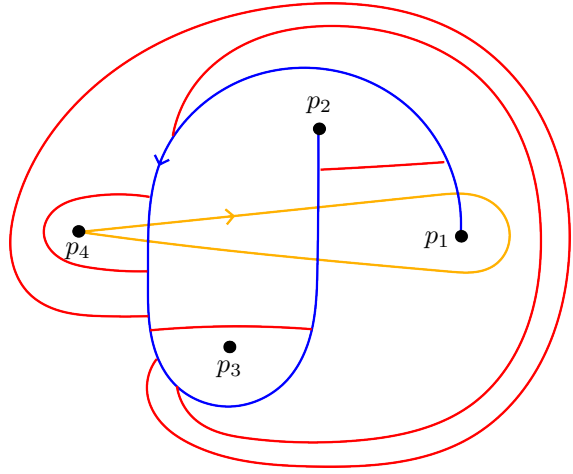
\end{proof} 

    The five possible types of disk that can arise up to equivalence are shown in Figure~\ref{figure:wn-disks-five} using our standard embedding of $\alpha$ and $\Gamma_1$ in $D_4$. 

    \begin{figure}[htpb!]
       \centering
        \begin{overpic}[width=0.75\linewidth]{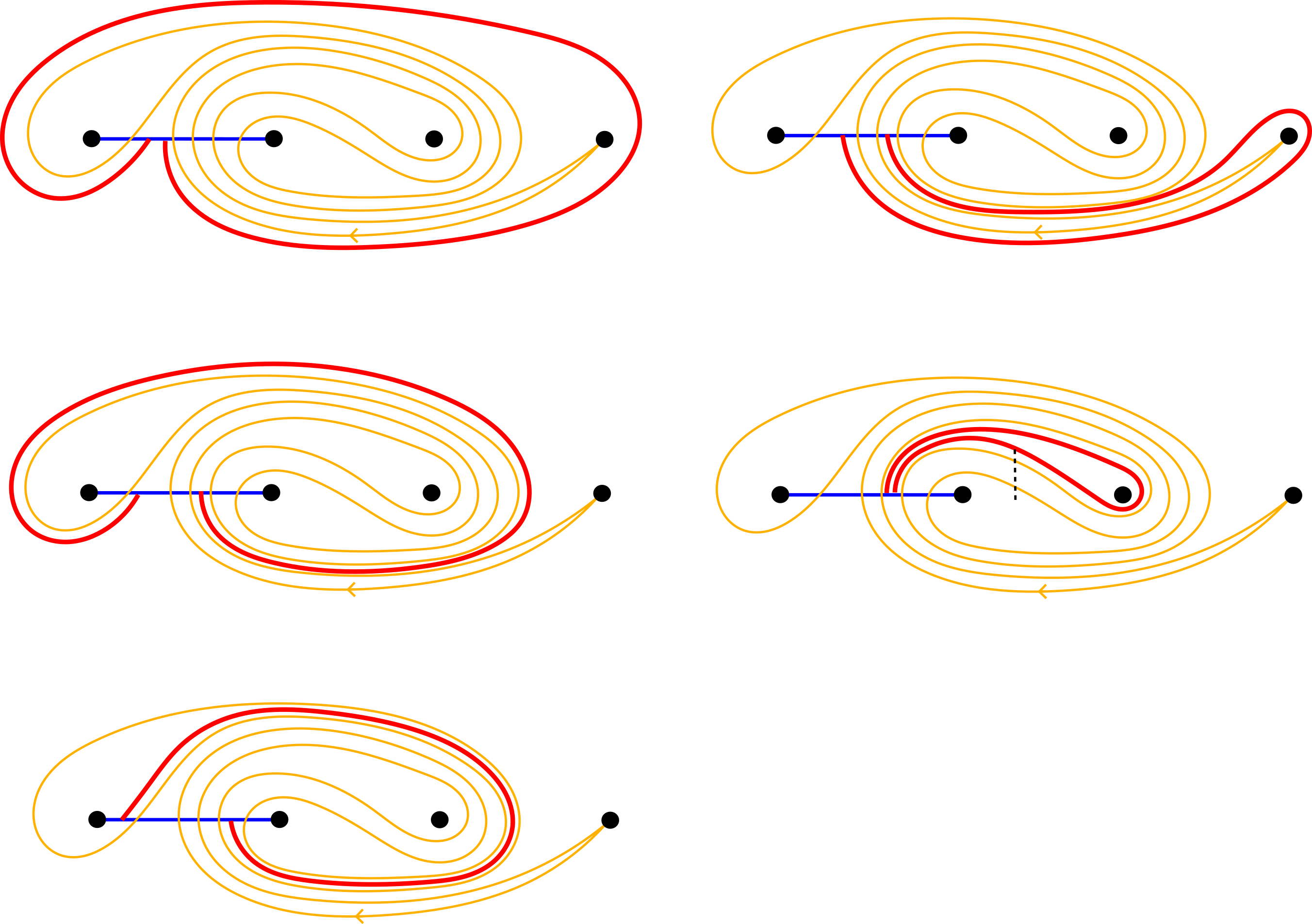}
            \put(3,69){\textcolor{red}{$\delta_1$}}
            \put(97,64){\textcolor{red}{$\delta_2$}}
            \put(3,41){\textcolor{red}{$\delta_3$}}
            \put(76,29.5){\textcolor{red}{$\delta_4$}}
            \put(8,11){\textcolor{red}{$\delta_5$}}
        \end{overpic}
        \caption{The loop $\Gamma_1$ is shown in gold, together with five red arcs $\delta_1, \ldots, \delta_5$ indicating all five types of disks arising in the disk sequence of a braid in $K_4$ such that the $\beta$-component of the disk can be isotoped off of $\Gamma_1$.}
        \label{figure:wn-disks-five}
    \end{figure}

\subsection {Embedding in ${\bm \B_5}$} 
\label{section:embedding}
Under the hypotheses of Lemma~\ref{lemma:only 4}, there is only one possible type of disk arising in the disk sequence of a braid in $K_4$ that violates the parity condition: a disk containing all four punctures.  In order to deal with this, we will pass briefly to the setting of the disk with five punctures.  To that end, we now fix an embedding of $D_4$ in the disk $D_5$; we can think of this as simply fixing an additional marked point $p_5$ in the interior of the disk $D_4$. If $\Phi \in B_4$, we let $f(\Phi)\in B_5$ denote the image of $\Phi$ in $B_5$ under the induced injection $f: \B_4 \hookrightarrow \B_5$. 

The main goal of Section~\ref{section:embedding} is to prove Proposition~\ref{prop:four-removal} and its corollary, which (roughly speaking) establish conditions under which a push-map will satisfy the parity condition.  We will begin with a series of lemmas that will equip us to deal with 4-disks, which under our hypotheses are the one remaining obstruction to satisfying the parity condition.  

Let $\beta$ be an arc in $D_5$ from $p_*$ to $p_3$, let $\Gamma$ be a loop based at $p_5$, and let $\Gamma_F$ denote the subarc of $\Gamma$ traveling from its final point of intersection with $\alpha$ back to $p_5$.   We note that the set of 4-disks in the disk sequence of $\Phi$ are necessarily nested; in particular, their respective $\beta$-components form a ``band'', as indicated in Figure~\ref{figure:4s become 5s both}.  We will say the loop $\Gamma$ is {\it disjoint from $\beta$ modulo 4-disks} if $\Gamma_F$ intersects this band of $\beta$-components transversely and $\Gamma$ is otherwise disjoint from $\beta$, so that $\Gamma$ intersects the $\beta$-component of each 4-disk precisely once and $\Gamma$ is disjoint from the $\beta$-component of every other disk in the disk sequence of $\beta$.

\begin{lemma}
\label{lemma:4s become 5s}
    Let $\Phi \in B_4$, let $\beta = (\beta_*^3) \Phi$, and let $\Gamma$ be a loop in $D_5$ based at $p_5$ that is disjoint from $\beta$ modulo 4-disks.  Then the disk sequence of $f(\Phi) \cdot \Gamma$ is obtained from the disk sequence of $\Phi$ by replacing each 4-disk with a 5-disk.  Moreover, each such 5-disk is necessarily sign-changing.
\end{lemma}

\begin{figure}[htpb!]
        \centering
        \vspace{2mm}
        \begin{subfigure}{.45\linewidth}
            \begin{overpic}[width=0.9\linewidth]{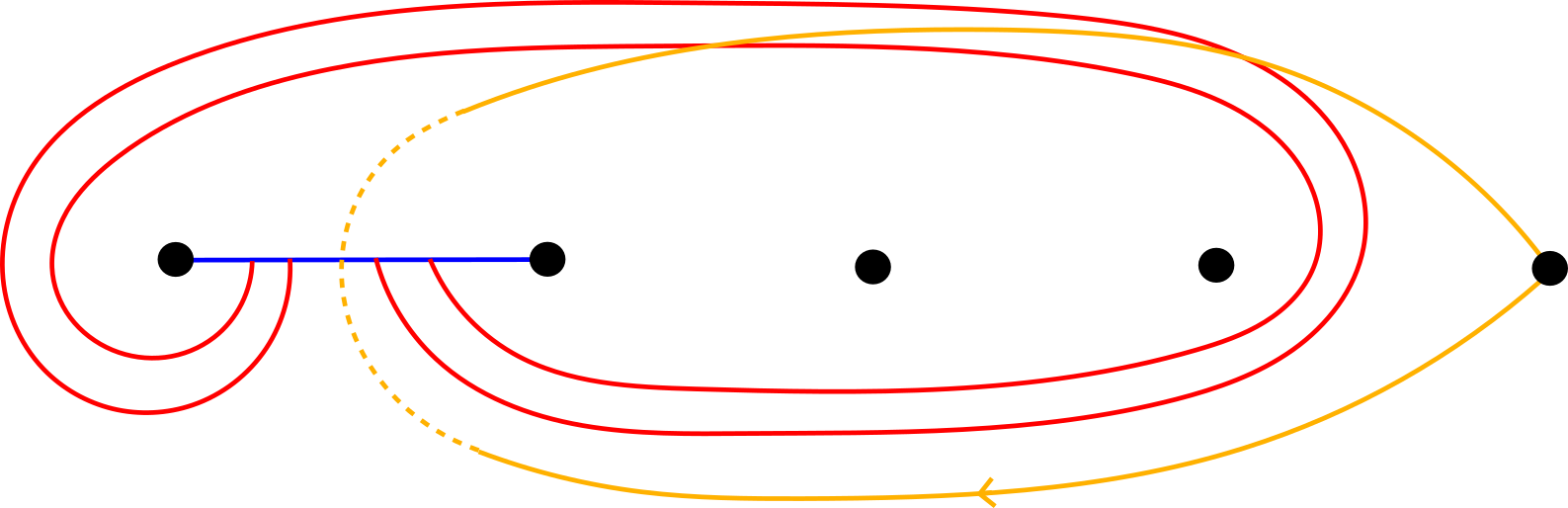}
                % \put(32, 50){\textcolor{blue}{$\alpha$}}
            \put(-5, 18){\textcolor{red}{$\beta$}}
           \put(95, 22){\textcolor[RGB]{255,177,0}{$\Gamma$}}
            \end{overpic}
        % \label{figure:4 disks become 5 disks}
        \end{subfigure}
        \begin{subfigure}{.45\linewidth}
        \includegraphics[width=0.9\linewidth]{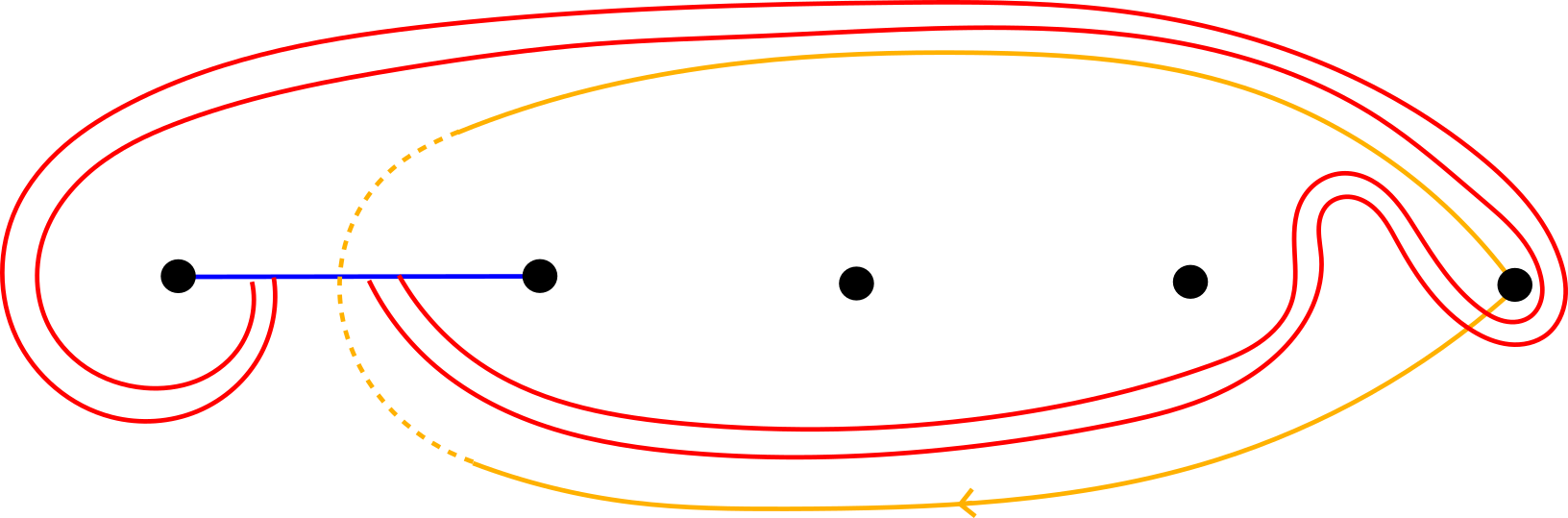}

        % \label{figure:4s become 5s part two}
        \end{subfigure}
        \caption{The loop $\Gamma$ is disjoint from $\beta$ except for intersecting the ``band'' of $\beta$-components of the nested 4-disks. After performing the push-map $\Gamma$, the band of 4-disks is replaced with a band of 5-disks.}
        \label{figure:4s become 5s both}
    \end{figure}

Before giving the proof, we emphasize that the loop $\Gamma$ in the statement of Lemma~\ref{lemma:4s become 5s} need not be simple.  

\begin{proof}
Since the only points of intersection between $\Gamma$ and $\beta$ lie in $\Gamma_F$, the effect on $\beta$ of pushing along $\Gamma$ is to push $p_5$ ``into'' each 4-disk, as shown in Figure~\ref{figure:4s become 5s both}, while any other disks remain unchanged.  As we have already observed, each 4-disk in the disk sequence of $\Phi$ is necessarily sign-changing, and hence the newly created 5-disks must also be sign-changing.
\end{proof}

As above, we consider a loop $\Gamma$ in $D_5$ based at $p_5$, and let $\Gamma_F$ denote the subarc of $\Gamma$ traveling from its final point of intersection with $\alpha$ back to $p_5$.  Also, we refer to a disk whose boundary is the union of a single subarc of $\alpha$ together with a single subarc of $\Gamma$, where $\Gamma$ is a loop, as an $\algam$ disk. Finally, we say that $\Gamma$ is {\it simple modulo 4-disks} if any self-intersection points of $\Gamma$ occur as a point of intersection of $\Gamma_F$ with the boundary of an $\algam$ 4-disk.

Recall that $\Gamma_I$ denotes the subarc of $\Gamma$ from $p_5$ to its first point of intersection with $\alpha$.  If $\Gamma$ is simple modulo 4-disks, we will say it is {\it properly simple modulo 4-disk} if the following three conditions hold: (1) each $\algam$ disk is equivalent to one of the five disk types shown in Figure~\ref{figure:wn-disks-five}, (2) there is some $\algam$ disk bounding a 3-disk containing $p_1$, $p_2$ and $p_3$, and (3) the arc $\Gamma_I$ does not intersect $\beta_*^3$.

\begin{lemma}
\label{lemma:alternating}
    Let $\Gamma$ be a loop in $D_5$ based at $p_5$ that is properly simple modulo 4-disks, let $q_1, \cdots, q_m$ denote the points of intersection of $\Gamma$ with $\beta_*^3$ in order as we travel along $\beta_*^3$, and let $\sign(q_i)$ denote the sign of the intersection of $\Gamma$ with $\beta_*^3$ at $q_i$.  Then for all $1 \leq i < m $, we have that either $\sign(q_i) = -\sign(q_{i+1})$, or else the corresponding disks $\Delta_i$ and $\Delta_{i+1}$ are both 4-disks.
\end{lemma}

In other words, Lemma~\ref{lemma:alternating} says that the sign of the intersection of $\Gamma$ with $\beta_*^3$ is alternating as we travel along $\beta_*^3$, unless the intersections correspond to two adjacent 4-disks; for convenience we refer to this property as being {\it alternating modulo 4-disks}.  

\begin{proof}
Since $\Gamma$ is properly simple modulo 4-disks, it follows that $p_5$ is not contained in any $\algam$ disks.  Hence we can view $p_5$ as a basepoint on the boundary of the disk.  

Let $q_i$ and $q_{i+1}$ be two points in $\Gamma \cap \beta_*^3$ that are adjacent along $\beta_*^3$, and suppose that $\sign(q_i) = \sign(q_{i+1})$.  Let $\beta_i$ denote the subarc of $\beta_*^3$ from $q_i$ to $q_{i+1}$, and similarly let $\gamma_i$ denote corresponding subarc of $\Gamma$ .  Without loss of generality, we may assume that $\beta_i \cup \gamma_i$ is a circle containing at most 3 punctures, contained in the embedded copy of $D_4 \subset D_5$.  

We now choose a point $x$ that lies along $\Gamma$ and between $q_{i+1}$ and $p_5$ if $\Gamma$ is simple, or otherwise between $q_{i+1}$ and the first point of self-intersection of $\Gamma$ as we travel along $\Gamma$; see Figure~\ref{figure:beta_i gamma_i}.  The point $x$ necessarily lies in the interior of $\beta_i \cup \gamma_i$.  

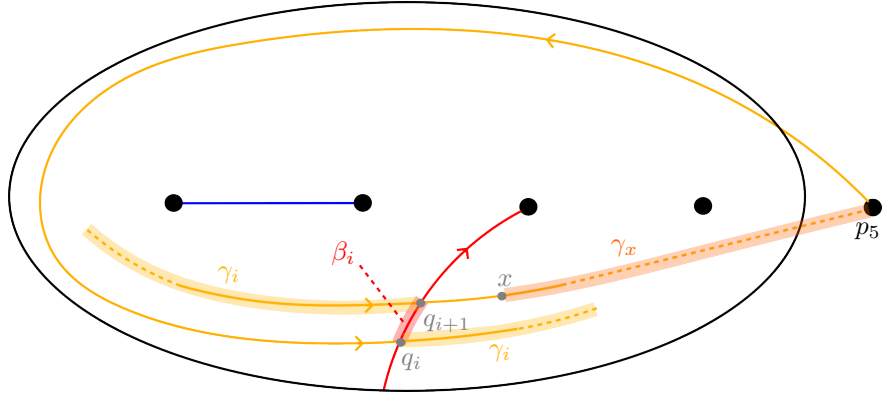
\begin{figure}[htpb!]
       \centering
        \begin{overpic}[width=0.7\linewidth]{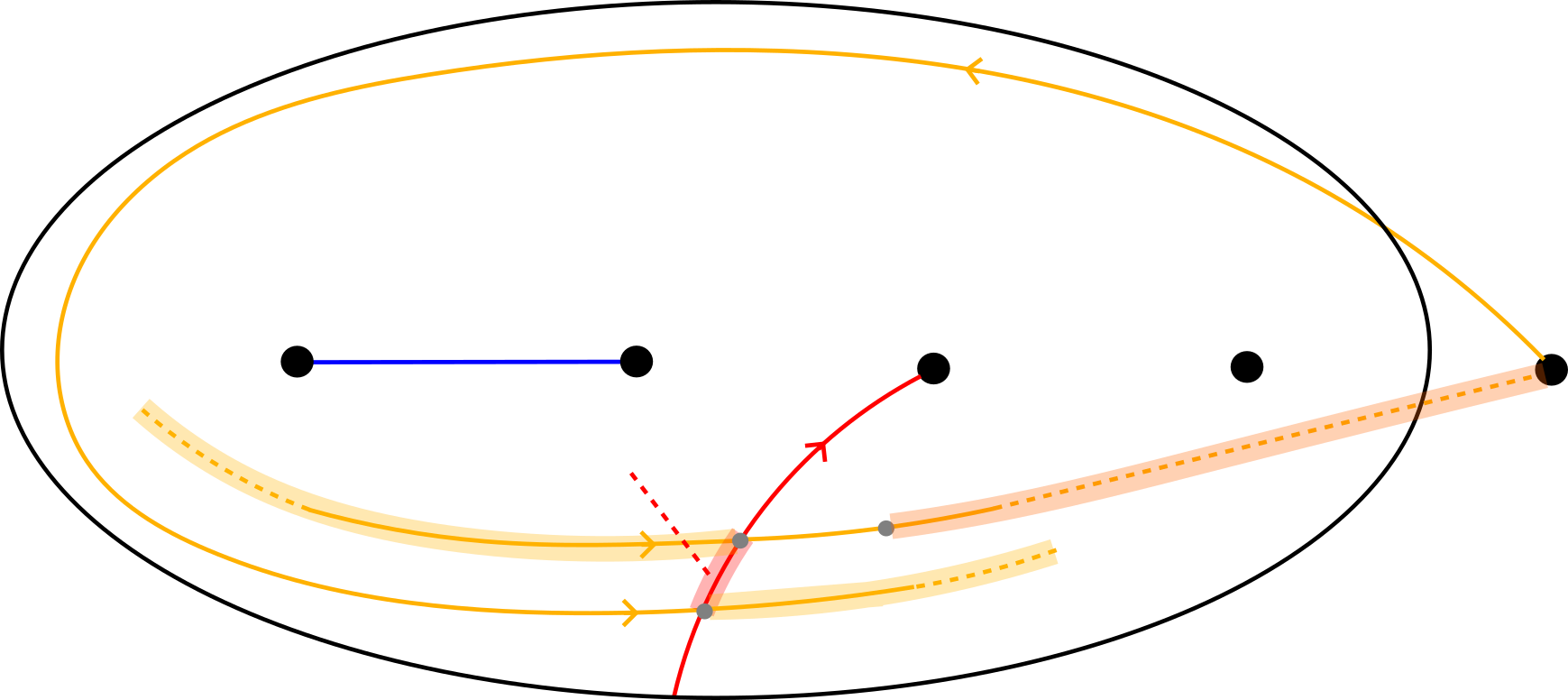}
   \put(24,13){\textcolor[RGB]{255,177,0}{$\gamma_i$}}
   \put(55,4){\textcolor[RGB]{255,177,0}{$\gamma_i$}}
      \put(69,16){\textcolor[RGB]{255,100,0}{$\gamma_x$}}
\put(37,15){\textcolor{red}{$\beta_i$}}
\put(45,3){\textcolor[RGB]{128,128,128}{$q_i$}}
\put(47.5,7.5){\textcolor[RGB]{128,128,128}{$q_{i+1}$}}
\put(56,12){\textcolor[RGB]{128,128,128}{$x$}}
\put(97,18){{$p_5$}}
        \end{overpic}
        \caption{The disk bounded by $\beta_i \cup \gamma_i$.} 
        \label{figure:beta_i gamma_i}
    \end{figure}

Now, let $\gamma_x$ denote the subarc of $\Gamma$ from $x$ to $p_5$.  Then $\gamma_x$ necessarily intersects $\beta_i \cup \gamma_i$ in at least one point, which we denote by $y$.  If $y \in \beta_i$ this contradicts our assumption that $q_i$ and $q_{i+1}$ are adjacent along $\beta_*^3$.  If $y \in \gamma_i$, then this contradicts our assumption that $\Gamma$ is properly simple modulo 4-disks.  This completes the proof that $\Gamma$ is alternating module 4-disks with respect to $\beta_*^3$.
\end{proof}

The next lemma will be crucial for identifying braids that satisfy the parity condition.   
\begin{lemma}
\label{lemma:properly simple modulo 4-disks}
    Let $\Gamma$ be a loop in $D_5$ based at $p_5$ that is properly simple modulo 4-disks. Then $\Gamma$ satisfies the parity condition.
    % and such that 
\end{lemma}

\begin{figure}[htpb!]
       \centering
        \begin{overpic}[width=0.25\linewidth]{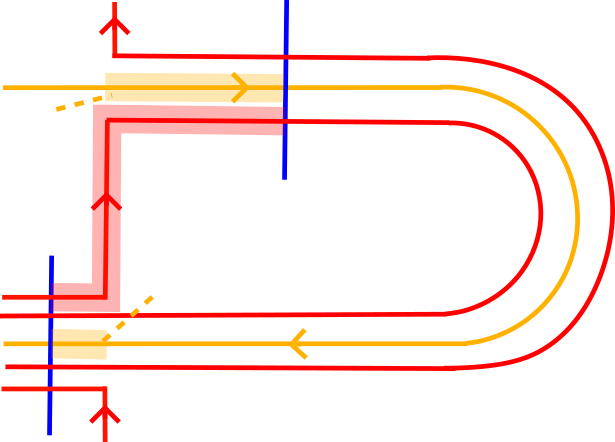}
          \put(4,48){\textcolor[RGB]{255,177,0}{$\gamma_1$}}
          \put(24,26){\textcolor[RGB]{255,177,0}{$\gamma_2$}} 
           \put(45,75){\textcolor{blue}{$\alpha$}}
           \put(7,-4){\textcolor{blue}{$\alpha$}}
           \put(17.5,75.5){\textcolor{red}{$\beta$}}
        \end{overpic}
        \caption{The case where $\Delta$ is formed using two new subarcs and one original subarc.}
        \label{figure:gamma1 gamma2}
    \end{figure}
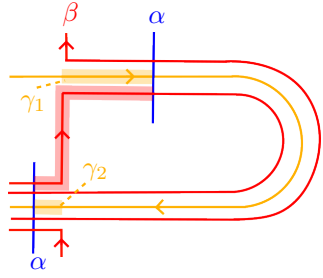

\begin{proof}
Consider a disk $\Delta$ in the disk sequence of $\Gamma$. In what follows we will use the terminology of {\it original} and {\it new} subarcs and vertices, as introduced in the proof of Proposition~\ref{prop:ypush_minimalposition}.

% In particular, consider the application of the push-map in stages, where ... 
  
First, we note that both vertices of the disk $\Delta$ are new.  Setting $\beta = (\beta_*^3) \Gamma$, we have moreover that the $\beta$-component of $\Delta$ is either a single new subarc, or it is the union of two new subarcs with an original subarc; see Figure~\ref{figure:gamma1 gamma2}.  

In the former case, there are three possibilities:  the disk $\Delta$ corresponds to an original $\algam$ disk that is not a 4-disk, $\Delta$ is a 1-disk containing $p_5$, or, by applying Lemma~\ref{lemma:4s become 5s}, $\Delta$ is a sign-changing 5-disk.  For each of these, $\Delta$ satisfies the parity condition. 

By part (3) of the definition of being properly simple modulo 4-disks, $\beta$ is the union of two new subarcs with an original subarc.  Let $\gamma_1$ and $\gamma_2$ denote the two subarcs of $\Gamma$ that correspond to the two new subarcs of $\beta$, where $\gamma_1$ precedes $\gamma_2$ in $\Gamma$.    Suppose first that $\gamma_1$ and $\gamma_2$ both occur as subarcs of an $\algam$ disks.  If $\gamma_1$ and $\gamma_2$ are subarcs of equivalent $\algam$ disks, then $\Delta$ corresponds to an original $\algam$-disk that is not a 4-disk, and hence $\Delta$ satisfies the parity condition. 

Suppose next that $\gamma_1$ and $\gamma_2$ are subarcs of inequivalent $\algam$ disks. We begin by observing that both $\gamma_1$ and $\gamma_2$ are necessarily oriented from $\beta_*^3$ to $\alpha$.  If neither $\gamma_1$ nor $\gamma_2$ occur as subarcs of a 4-disk, then $\Delta$ does not contain $p_5$.  Morever, $\Delta$ is equivalent to one of the disks shown in Figure \ref{figure:wn-disks-five} or is another sign-preserving 2-disk and hence satisfies the parity condition.  

Referring again to Figure~\ref{figure:gamma1 gamma2}, we note that if $\gamma_1$ occurs as a subarc of a 4-disk, then $\gamma_2$ must occur as part of an $\algam$ disk that is adjacent to the 4-disk as we travel along $\beta_*^3$.     

Since there is a 3-disk containing $p_1, p_2, p_3$ occurring as an $\algam$ disk, and since no self-intersections of $\Gamma$ can occur in this 3-disk, we see that $\gamma_2$ occurs either part of such a 3-disk or as part of a 1-disk containing $p_4$; see, respectively, the arcs $\delta_3$ and $\delta_2$ in Figure~\ref{figure:wn-disks-five}. By Lemma \ref{lemma:alternating}, the possible configurations that can occur are shown in Figure~\ref{fig:cases-fillers}.  Hence $\Delta$ is one of the following: a 1-disk containing $p_4$, a 3-disk containing $p_1, p_2,$ and $p_3$,  or a sign-changing 5-disk; each of these satisfy the parity condition, and the proof is complete. \end{proof}

\begin{figure}[htpb!]
    \centering
    \begin{overpic}[width=1\linewidth]{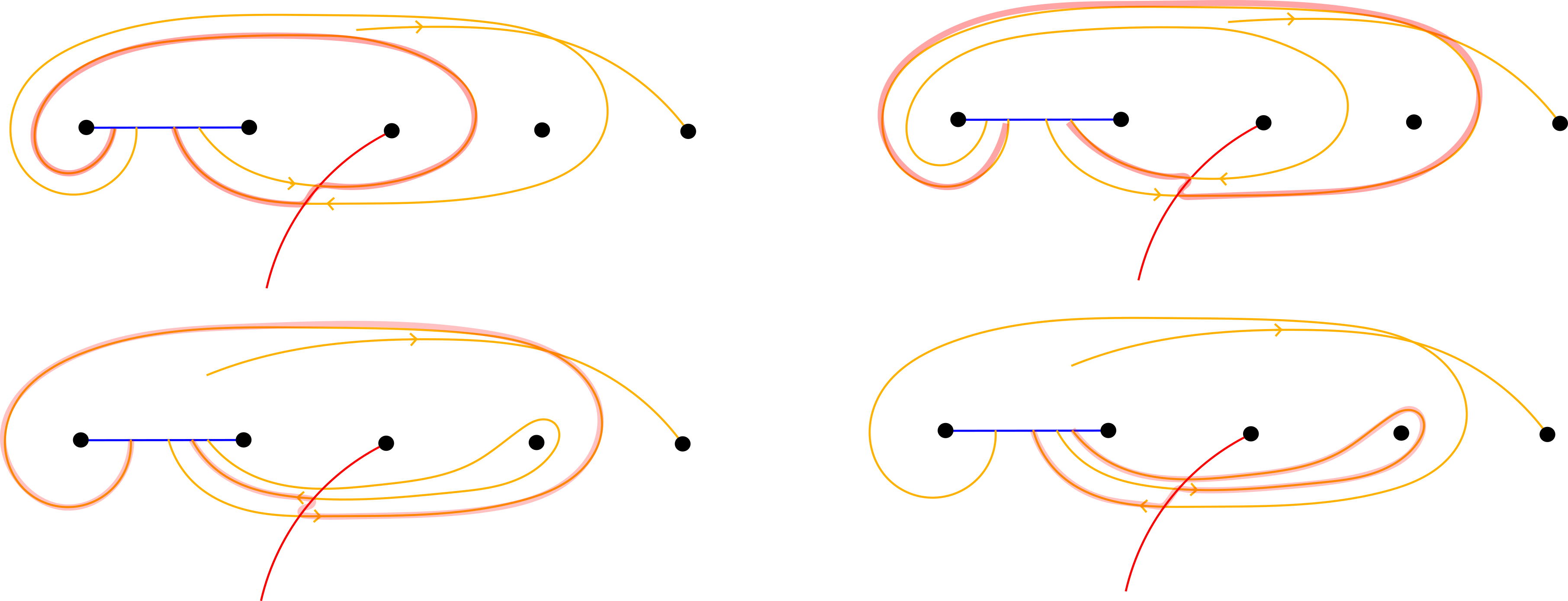}
          % \put(4,48){\textcolor[RGB]{255,177,0}{$\gamma_1$}}
          % \put(24,26){\textcolor[RGB]{255,177,0}{$\gamma_2$}} 
           % \put(45,75){\textcolor{blue}{$\alpha$}}
           % \put(7,-4){\textcolor{blue}{$\alpha$}}
           % \put(17.5,75.5){\textcolor{red}{$\beta$}}
        \end{overpic}
    \caption{The four possible configurations: in each case $\gamma_1$ and $\gamma_2$ are highlighted in red, joined by an original subarc of $\beta$.}
    \label{fig:cases-fillers}
\end{figure}

We emphasize that throughout what follows, the notation $\Gamma_1$ always refers to the particular push-map shown in Figure~\ref{figure:gamma1}, while $\Gamma$ refers to a push-map that generally depends on some other element $\Phi \in K_4$.

\begin{proposition}
\label{prop:four-removal}
    Let $\Phi\in \B_4$, and suppose that $\Phi$ can be written as a proper product $\Phi = \Phi' \cdot \Gamma_1$.  Then there exists a push-map $\Gamma \in K_5$ such that $f(\Phi) \cdot \Gamma$ is a proper product in $\B_5$ and such that both $f(\Phi) \cdot \Gamma$ and $\Gamma$ satisfy the parity condition. Furthermore we may choose $\Gamma$ so that $\iota(\beta_*^3\cdot (\Phi \cdot \Gamma), \alpha) \neq \iota(\beta_*^3 \cdot \Gamma, \alpha)$.
\end{proposition}

Before we begin the proof of Proposition~\ref{prop:four-removal}, we recall that an $\alpha$-$\gamma$ disk is a disk whose boundary is the union of a single subarc of $\alpha$ together with a single subarc of $\gamma$, where $\gamma$ is a loop or an arc.  In the case where $\gamma$ is a union of two or more subarcs of a larger loop or arc $\Gamma$, we will also refer to such a disk as an $\algam$ disk for convenience.

\begin{proof}
Suppose that the braid $\Phi \in \B_4$ does not satisfy the parity condition.  By Lemma~\ref{lemma:only 4} there must be at least one disk $\Delta$ in its disk sequence with $\mcP(\Delta) = \{ p_1, p_2, p_3, p_4 \}$.  Moreover, any 4-disks that occur in its disk sequence are necessarily nested in $D_4$.  Consider the image $f(\Phi)$ of our braid in $\B_5$.  We will construct a push-map $\Gamma \in K_5$ whose effect will be to replace each of the 4-disks in the disk sequence of $f(\Phi)$ with a 5-disk while preserving every other disk.

Let $\beta$ denote the arc $(\beta_*^3) f(\Phi)$ in the disk $D_5$, and choose a point $q$ in the interior of the innermost 4-disk of $\beta$ such that $q$ is not contained in a $k$-disk in the disk sequence of $f(\Phi)$ for any $k \leq 3$.  Since the arc $\beta$ is simple, its complement in the disk $D_5$ is connected, and hence we can choose a simple path $\gamma_1$ from $p_5$ to $q$ such that $\iota(\gamma_1, \beta) = 0$; see Figure~\ref{fig:construction-example} for an example.  We note that $\gamma_1$ necessarily intersects $\alpha$, since $q$ lies in the interior of the disk $\Delta$ and $p_5$ does not. 

\begin{figure}[htpb!]
    \centering
    \begin{overpic}[width=0.7\linewidth]{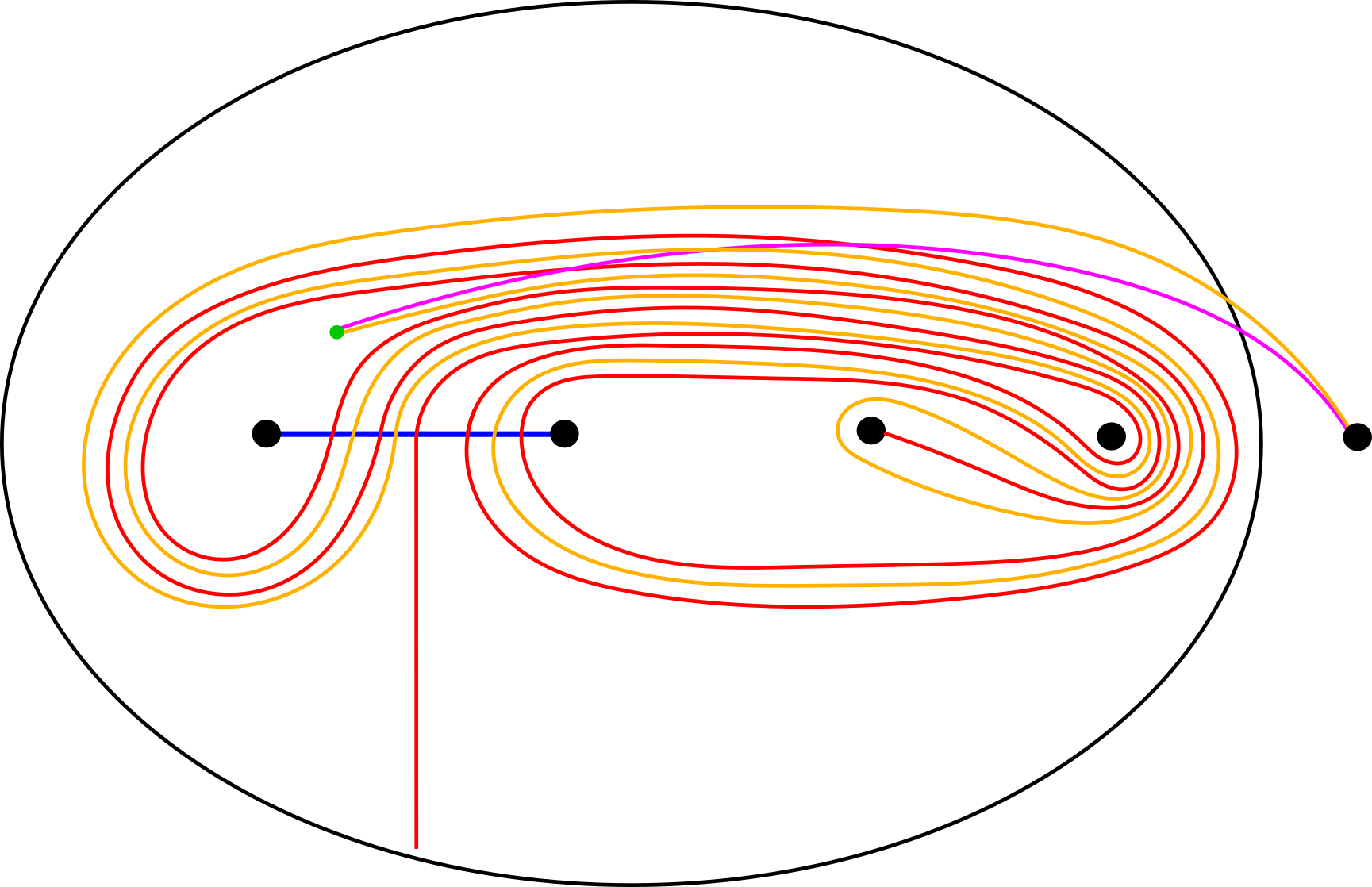}
                \put(98,28){$p_5$}
                \put(78,4){$D_4\subseteq D_5$}
 \put(31, 15){\textcolor[RGB]{200,0,0}{$\beta$}}
                \put(22, 39){\textcolor[RGB]{0,200,0}{$q$}}
                \put(92.5, 35){\textcolor[RGB]{255,0,255}{$\gamma_2$}}
                \put(65, 27){\textcolor[RGB]{255,177,0}{$\gamma_1$}}
    \end{overpic}
    
    \caption{An example of the construction of the loop $\Gamma = \gamma_1 \cup \gamma_2$.}
    \label{fig:construction-example}
\end{figure}
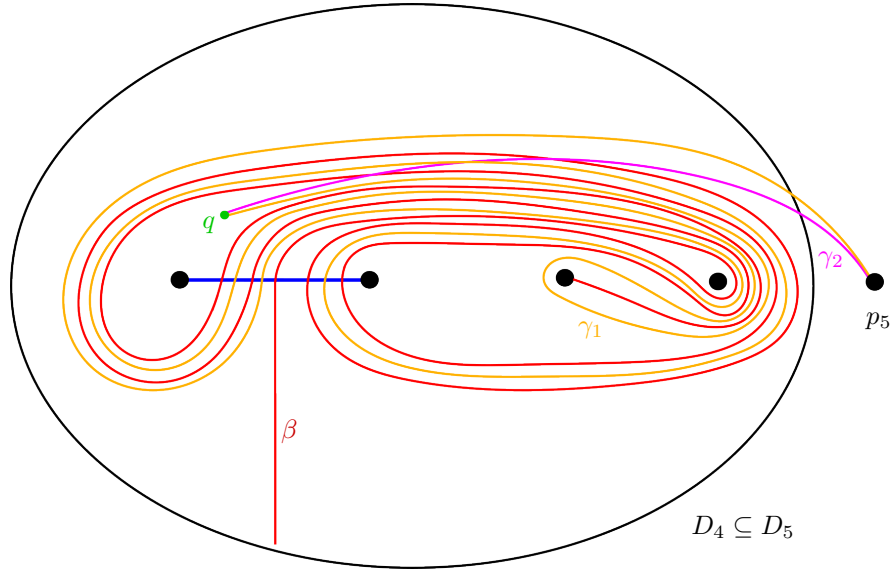

Next we choose a simple path $\gamma_2$ from $q$ to $p_5$ that is disjoint from $\alpha$ and not isotopic to $\gamma_1$, such that the intersection of $\gamma_2$ and $\beta$ consists of precisely one point in the $\beta$-component of each 4-disk in the disk sequence of $\Phi$. In other words, $\gamma_2$ travels transversely through the $\beta$-component of each 4-disk and then passes out of $D_4$ to $p_5$ without intersecting $\beta$ again, and in particular without intersecting any other disks in the disk sequence of $\Phi$; again, see Figure~\ref{fig:construction-example} for an example. Finally, we define the loop $\Gamma \in K_5$ to be the product of paths $\gamma_2 \circ \gamma_1$.  

By Lemma~\ref{lemma:4s become 5s}, the disk sequence of $(\beta) \Gamma = (\beta_*^3) (\Phi \cdot \Gamma)$ is exactly the same as that of $\beta = (\beta_*^3) \Phi$, except that each 4-disk in the original disk sequence of $\Phi \in \B_4$ has been replaced by a 5-disk that is necessarily sign-changing.  Hence $f(\Phi) \cdot \Gamma \in \B_5$ satisfies the parity condition.  

Consider next the disk sequence of $(\beta_*^3) \cdot \Gamma$. Note that $\Gamma$ is properly simple modulo 4-disks; condition (1) is immediately satisfied by the constriction, and condition (2) is satisfied as well from the fact that any proper product $\Phi$ yields a disk sequence containing a 3-disk bounding $p_1,p_2,p_3$. By Lemma \ref{lemma:properly simple modulo 4-disks}, this $\Gamma$ satisfies the parity condition as well. 

% We claim that any disk in this sequence is either a 5-disk, a 1-disk containing $p_5$, or is equivalent to a disk formed with $\delta_i$ for some $i \geq 2$ as shown in Figure~\ref{figure:wn-disks-five}; it then follows that $\Gamma$ also satisfies the parity condition.  

% To see this, we first note that by construction $\gamma_1$ is a simple arc, and we may assume that it admits no bigon-forming polygons with $\beta_*^3$ and $\alpha$.  Since $\gamma_1$ is disjoint from $\beta = (\beta_*^3) (\Phi)$, and since $\Phi$ can be written as a proper product of the form $\Phi' \cdot \Gamma_1$, we may further assume without loss of generality that each $\alpha$-$\gamma_1$ disk is equivalent to one of the five disks shown in Figure~\ref{figure:wn-disks-five}; otherwise we could have made a different choice of $\gamma_1$.  It follows that any $\alpha$-$\beta$ disks that result from pushing $\beta_*^3$ along $\gamma_1$ will also be equivalent to one of the five disks shown in Figure~\ref{figure:wn-disks-five}, or else it is a 1-disk containing only $p_5$.  In order to complete the push around $\Gamma$, we next push along $\gamma_2$.  As noted above, the effect of doing so will be to replace any existing 4-disks formed in the first stage with 5-disks.  The claim follows, and we have established the first statement of the proposition.

It remains to show that we can always choose $\Gamma$ to ensure the following statement holds:
\begin{equation}
\label{equation:not equal}
    \iota(\alpha, (\beta_*^3)\cdot\Gamma) \neq \iota (\alpha, (\beta_*^3)(\Phi\cdot\Gamma))
\end{equation}

Suppose now that in our construction of $\Gamma$ as above, we have $\iota(\alpha, (\beta_*^3) \cdot \Gamma) = \iota (\alpha, (\beta_*^3)(\Phi \cdot \Gamma))$, or in other words, $\iota(\alpha, (\beta_*^3) \Gamma) = \iota (\alpha, (\beta) \Gamma)$.  Our strategy is to alter our construction of the loop $\Gamma$ slightly so that the left-hand side of this equality increases, while the right-hand side remains unchanged.  

To this end, we will form a new loop $\Gamma'$ starting with the same choice of $\gamma_1$ as before, but we will choose a different arc $\gamma_2'$ joining the point $q$ to $p_5$ to replace $\gamma_2$, and set $\Gamma' = \gamma_1 \cup \gamma_2'$.  We note that it is sufficient to find such an arc $\gamma_2'$ satisfying the following two properties:
\begin{enumerate}
\item as we traverse $\gamma_2'$ from $q$ to $p_5$, its first point of intersection with $\beta = (\beta_*^3)\cdot(\Phi)$ does not occur until after its final point of intersection with $\alpha$; and
\item $\iota(\alpha, (\beta_*^3) \cdot \Gamma') > \iota(\alpha, (\beta_*^3) \cdot \Gamma)$.
\end{enumerate}
This follows from the fact that the first condition implies that $\iota (\alpha, (\beta) \Gamma') = \iota (\alpha, (\beta) \Gamma)$.  

Revisiting the construction of $\Gamma$, let $s$ denote the final point of intersection of $\gamma_1$ with $\alpha$ as we travel along $\gamma_1$, and let $(\gamma_1)_F$ denote the subarc of $\gamma_1$ joining $s$ to the point $q$.  Since $D_4 \backslash (\beta \cup \gamma_1)$ is connected, there exists an arc $\gamma'$ joining $q$ to a point $r \in \alpha$, where the point $r$ depends on the choice of $\gamma'$, such that $\gamma'$ is disjoint from $\beta$ and such that $\gamma' \cap \gamma_1 = \{ q \}$.  

We will now consider two cases.  First, suppose that $\gamma'$ can be chosen so that, the arc $\gamma' \cup (\gamma_1)_F$, together with the subarc of $\alpha$ joining $r$ to $s$, forms a disk $\Delta$ containing $p_1, p_2, p_3$, and $p_4$.  

    \begin{figure}[htpb!]
        \centering
            \begin{overpic}[width=0.7\linewidth]{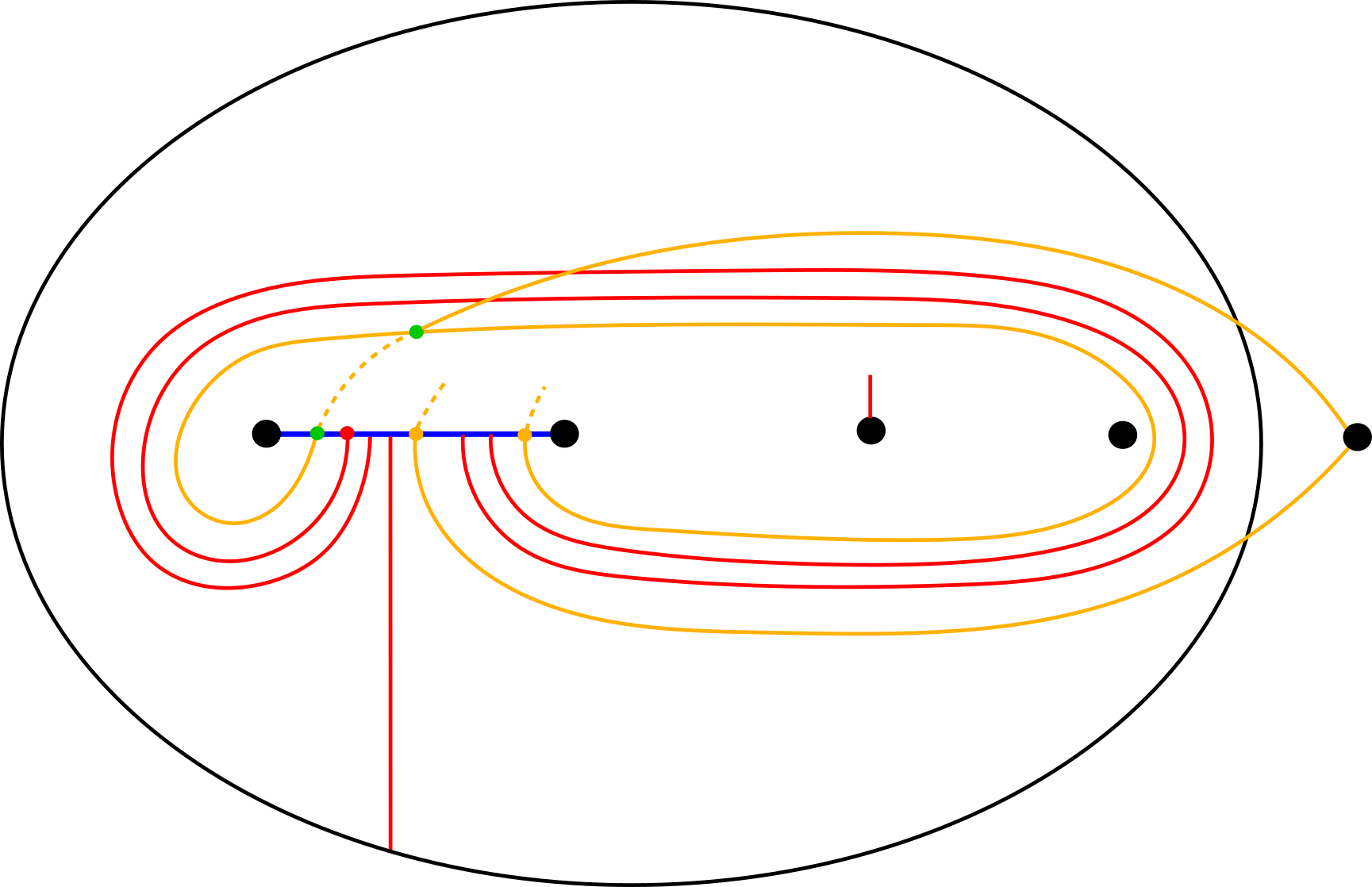}
                \put(98,28){$p_5$}
                \put(78,4){$D_4\subseteq D_5$}
                \put(29, 38){\textcolor[RGB]{0,200,0}{$q$}}
                \put(36, 34){\textcolor[RGB]{255,177,0}{$s$}}
                \put(21, 34){\textcolor[RGB]{0,200,0}{$r$}}
                \put(22, 37){\textcolor[RGB]{255,177,0}{$\gamma''$}}
                \put(14, 30){\textcolor[RGB]{255,177,0}{$\gamma'$}}
                \put(60, 49){\textcolor[RGB]{255,177,0}{$\gamma_2$}}
                \put(65, 27){\textcolor[RGB]{255,177,0}{$(\gamma_1)_F$}}
            \end{overpic}
        \caption{An alternate construction of $\gamma_2$: Case 1.}
        \label{figure:adjusting gamma_2 case 1}
    \end{figure}
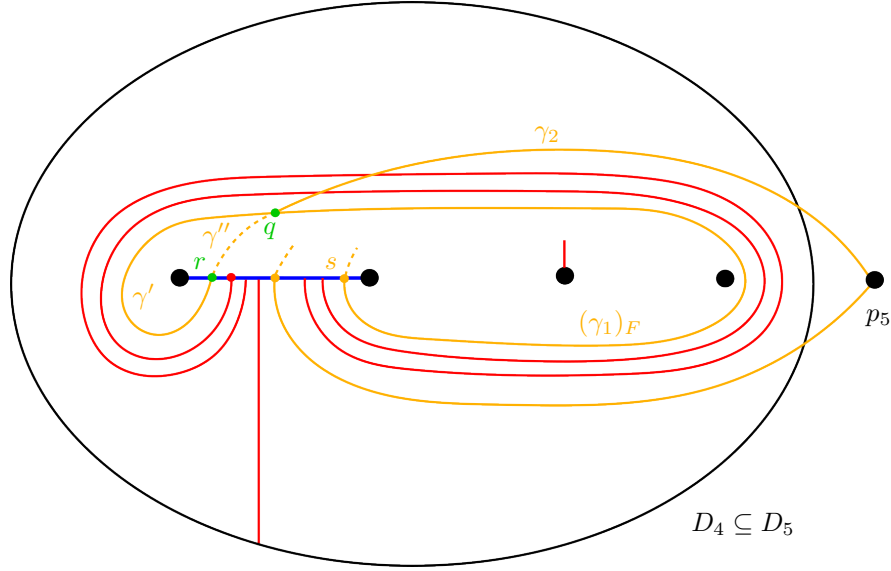

We next look for an arc $\gamma''$ from $r$ back to $q$ satisfying the following:
\begin{enumerate}
\item the arc $\gamma''$ is disjoint from $\beta$;
\item $\gamma'' \cap (\gamma_1 \cup \gamma') = \{q, r \}$; and 
\item $\gamma''$ is not isotopic to $\gamma'$ relative to the endpoints $\{q, r \}$.
\end{enumerate}
This is equivalent to finding a nontrivial loop based at $r$ in $D_4 \backslash (\beta \cup \gamma_1 \cup \gamma')$.  This is always possible since $\beta \cap (\gamma_1 \cup \gamma') = \emptyset$, and hence $D_4 \backslash (\beta \cup \gamma_1 \cup \gamma')$ is homeomorphic to a disk with three punctures: $p_1, p_2,$ and $p_4$.  Finally, we set $\gamma_2' = \gamma' \cup \gamma'' \cup \gamma_2$, and set $\Gamma' = \gamma_1 \cup \gamma_2'$.  Since the subarc $\gamma_2'$ satisfies the two properties above, we have established the result in this case.

It remains to deal with the case where the arc $\gamma'$ joining $q$ to a point $r \in \alpha$ cannot be chosen so that $\gamma' \cup (\gamma_1)_F$ forms a 4-disk, and hence forms a disk containing a proper subset of $\{ p_1, p_2, p_3, p_4 \}$; recall that we are still using in this case that $\gamma'$ is disjoint from $\beta$ and that $\gamma' \cap \gamma_1 = \{ q \}$. 

We will choose the arc $\gamma'$ to travel from $q$ to a point $r \in \alpha$ in parallel to the $\beta$-component of the boundary of the innermost 4-disk; there are two possibilities here.   One choice will result in $(\gamma_1)_F \cup \gamma'$ forming an $\algam$ disk containing $k$ punctures, where $k < 4$, and the other choice for $\gamma'$ results in an $\algam$ disk containing $n - k$ punctures.  Referring to Figure~\ref{figure:adjusting gamma_2 case 2}, we see that here $\gamma'$ is chosen to follow the innermost such $\beta$-component in the right-hand direction (as shown in the figure), so that $(\gamma_1)_F \cup \gamma'$ forms a 1-disk with $\alpha$; following to the left would form a 3-disk instead.  In other words, we can assume without loss of generality that the $\algam$ disk formed by $(\gamma_1)_F \cup \gamma'$ contains at most two punctures.
    \begin{figure}[htpb!]
        \centering
        \vspace{2mm}
            \begin{overpic}[width=0.6\linewidth]{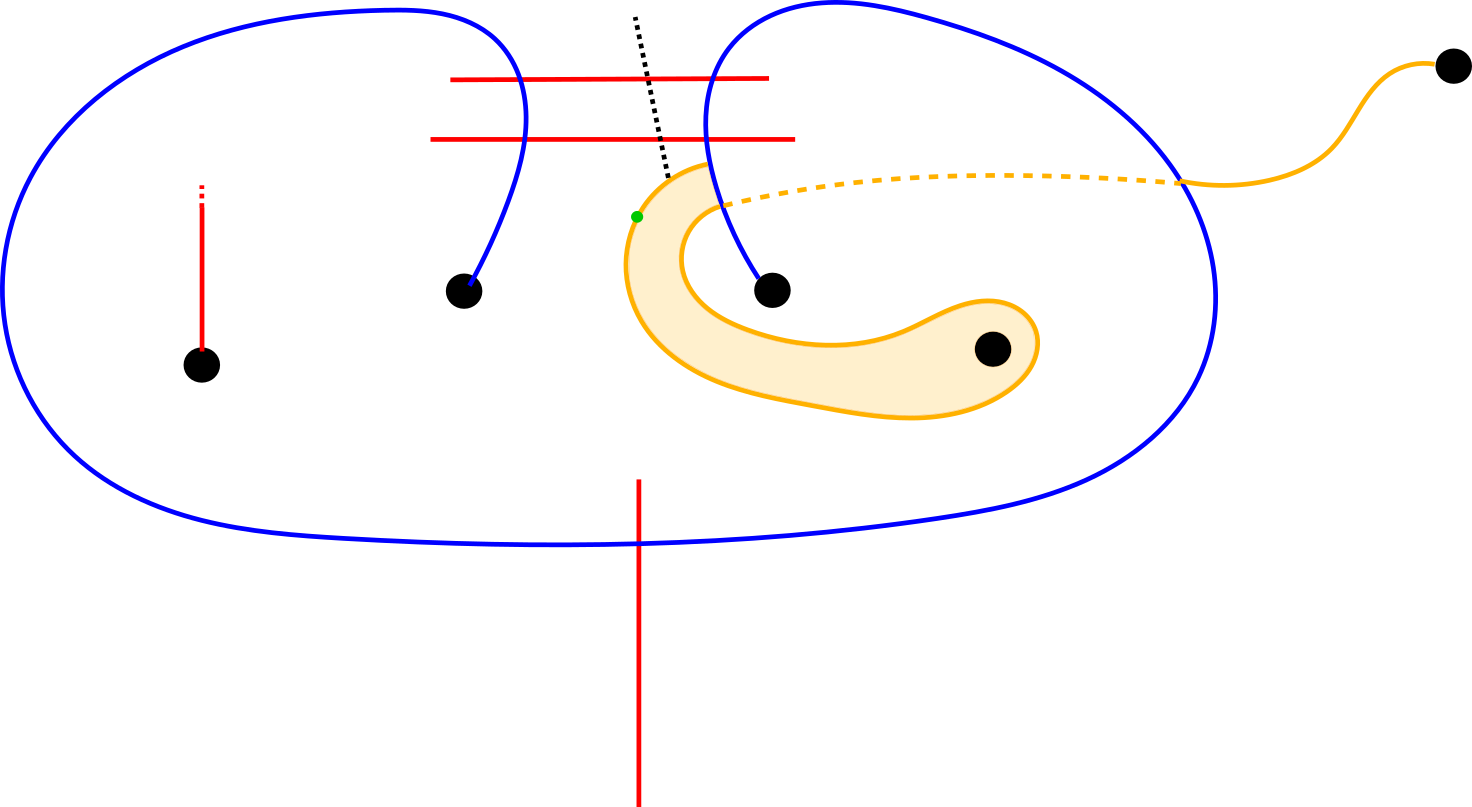}
            \put(40,40){\textcolor[RGB]{0,200,0}{\small $q$}}
            \put(41,56){\textcolor[RGB]{255,177,0}{\small $\gamma'$}}
            \put(50,22){\textcolor[RGB]{255,177,0}{\small $(\gamma_1)_F$}}
            \end{overpic}
        \caption{The horizontal lines in this figure represent $\beta$-components of 4-disks.  The choice of the arc $\gamma'$ shown here yields a disk with 1 puncture, by traveling to the right in parallel with the $\beta$-component of the innermost 4-disk.   }
        \label{figure:adjusting gamma_2 case 2}
    \end{figure}

Our construction now reduces as in the previous case to finding a nontrivial loop based at $r$ in $D_n \backslash (\beta \cup \gamma_1 \cup \gamma' \cup \gamma'')$.  Again, this is always possible since $\beta \cap (\gamma_1 \cup \gamma') = \emptyset$, and since by construction $D_4 \backslash (\beta \cup \gamma_1 \cup \gamma')$ is homeomorphic to a disk with at least one puncture.  This completes the proof of the proposition.
\end{proof}

We remark that the same construction used to find the loop $\Gamma$ in the proof of Proposition~\ref{prop:four-removal} could be iterated to find, for any $n$, and for any $\Phi\in B_n$, some $\Gamma_1,\dots,\Gamma_k$ so that, for the appropriate inclusion map $f$, we have that $f(\Phi) \cdot \Gamma_1 \dots \Gamma_k \in B_{n+k}$ satisfies the parity condition.   The rough idea here is to start with the innermost disk of each occuring type that does not satisfy the parity condition, and find a suitable push-map that ``corrects'' its parity.  However, in this case we cannot guarantee that the braid $\Gamma_1 \dots \Gamma_k$ satisfies the parity condition.  In our case, the fact the braid $\Gamma$ also satisfies the parity condition is important in the proof of our Main Theorem.  The fact that we only need to correct one type of disk, so we only require one push map $\Gamma$, is crucial.  In the general case where $k > 1$, we would not be able to simultaneously assume that $\Gamma_2$ does not cross both $\beta$ and $(\beta)\cdot\Gamma_1$, and so this argument would fail. 

The following corollary of Proposition~\ref{prop:four-removal} enables us to derive conditions on the Moody polynomial of a push-map $\Phi \in K_4$, as part of a proper product.  

\begin{corollary}\label{corollary:unequal-moody}
    If $\Phi \in K_4$ admits a proper product $\Phi = \Phi' \cdot \Gamma_1$ and satisfies $\iota(\alpha, (\beta_*^3) \Phi) > 0$, then there exists a push-map $\Gamma \in K_5$ such that $\Moody_{f(\Phi) \cdot \Gamma} \neq \Moody_{\Gamma}$.
\end{corollary}
\begin{proof}

 Let $\Phi$ be a braid in $K_4$, and let $\Gamma$ be an associated push-map in $K_5$ as in the statement of Proposition~\ref{prop:four-removal}.  By Proposition~\ref{prop:four-removal}, both $\Gamma$ and $\Phi \cdot \Gamma$ satisfy the parity condition.  It then follows from Lemma~\ref{lemma:no-cancellations} that there are no cancellations in either of the corresponding Moody polynomials $\Moody_{\Gamma}$ and $\Moody_{f(\Phi)\cdot\Gamma}$. 

Let $M_-$ and $M_+$ be the lowest and highest exponents that appear in the Moody polynomial $\Moody_{f(\Phi)\cdot\Gamma}$, and similarly let $M_-'$ and $M_+'$ denote the corresponding values for $\Moody_{\Gamma}$. Then, since neither polynomial has any cancellations, we can write the corresponding Moody polynomials as follows:
\begin{eqnarray*}
    \Moody_{f(\Phi)\cdot\Gamma} = \sum_{j=M_-}^{M_+} a_j t^j & = &|a_{M_-}|+\dots+|a_{M_+}| = \iota(\alpha, (\beta_*^3)(f(\Phi)\cdot\Gamma));  \qquad \mbox{ and } \\
    \Moody_{\Gamma} = \sum_{j=M'_-}^{M'_+} b_j t^j &=& |b_{M_-'}|+\dots+|b_{M'_+}| = \iota(\alpha, (\beta_*^3) \Gamma).
\end{eqnarray*}  
By Proposition \ref{prop:four-removal} we can assume that $\iota(\alpha, (\beta_*^3)(f(\Phi)\cdot\Gamma)) \neq \iota(\alpha, (\beta_*^3) \Gamma)$.  The result follows.
\end{proof}

\subsection{Faithfulness on ${\bm \Brun_4}$}
As described in Section~\ref{section:intro}, the proof of our Main Theorem follows from the next result.  

\begin{theorem}\label{theorem:k4-faithful}
  The Burau representation $\rho_4$ is faithful on its restriction to $\Brun_4$.
\end{theorem}
\begin{proof}
Let $\Phi$ denote a nontrivial element in $\Brun_4$.  We  make the following elementary observation: $\Phi \in \Brun_4$ lies in the kernel of the Burau representation $\rho_4$ if and only if, for every element $y \in \B_4$, the conjugate $y \Phi y^{-1}$ also lies in the kernel of $\rho_4$.  We also recall from our discussion in Section~\ref{section:intro} that any nontrivial element of $\Brun_4$ is pseudo-Anosov, and that the image of the arc $\beta_*^3$ under a pseudo-Anosov map necessarily intersects the arc $\alpha$ nontrivially.  

Given a simple loop such as $\Gamma_1$ the only possible bigon-forming polygons that can occur are $\beta$-to-$\alpha$ trigons and rectangles.  Now, $\beta$-to-$\alpha$ trigons can be resolved by an isotopy of $\beta$, and rectangles are not possible in this case since the points of intersection between $\alpha$ and $\Gamma_1$ alternate in sign as we travel along $\alpha$. 

Furthermore, if we consider $\Phi$ and $\Gamma_1$ as freely reduced words in the free group $K_4$ with respect to the free basis shown in Figure~\ref{figure:free-basis}, and if the juxtaposition of the two words (yielding the product $\Phi \cdot \Gamma_1$) is also a freely reduced word with respect to this free basis, then the second condition for $\Phi \cdot \Gamma_1$ to be a proper product is immediately satisfied.  One way to see this is to note that, as a freely reduced word in this basis, $\Gamma_1$ begins with $y_2^{-1}$.  If $\Phi \cdot \Gamma_1$ is freely reduced, then $\Phi$ cannot end with $y_2$; if it does, we can replace $\Phi$ with an appropriate conjugate that does not end with $y_2$.  Thus we may assume that $\Phi$ has been conjugated so that it can be written as a proper product $\Phi = \Phi' \cdot \Gamma_1$ for some $\Phi' \in K_4$.

\begin{figure}[htpb!]
    \centering
    \begin{overpic}[width=0.3\linewidth]{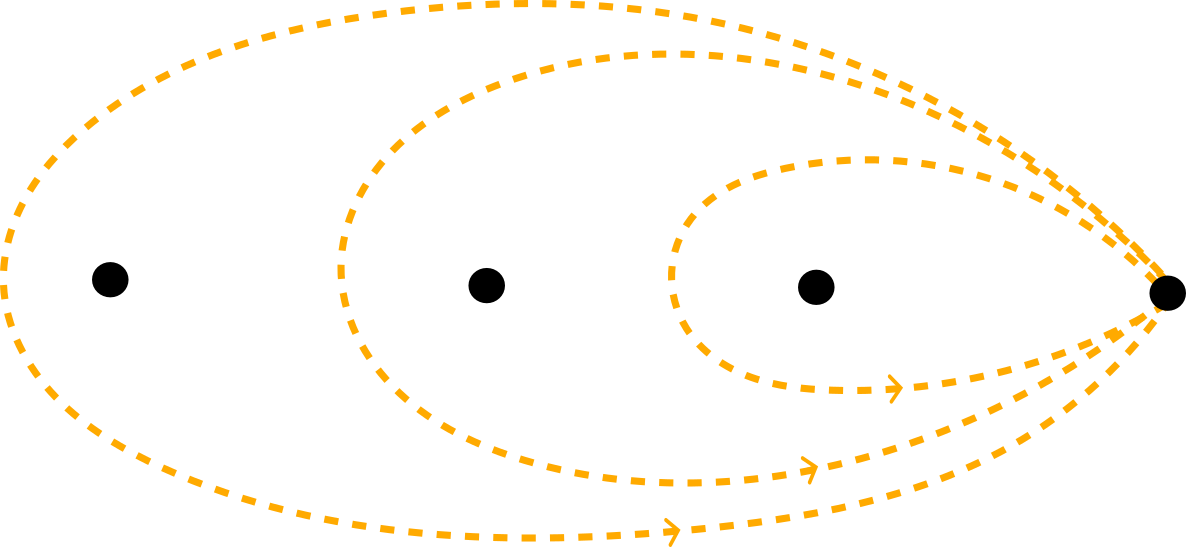}
        \put(54,33){$y_1$}
        \put(22,33){$y_2$}
        \put(-6.5,33){$y_3$}
    \end{overpic}
    \caption{A free basis for $K_4$.}
    \label{figure:free-basis}
\end{figure}

By Corollary~\ref{corollary:unequal-moody}, there is some push-map $\Gamma \in K_5 < \B_5$ so that $\Moody_{f(\Phi) \cdot \Gamma} \neq \Moody_{\Gamma}$. It now follows from Theorem~\ref{theorem:moody2} that $f(\Phi)$ does not lie in the kernel of $\rho_5$, and hence $\Phi$ does not lie in the kernel of $\rho_4$.  Hence the Burau representation of $B_4$ is faithful on the point-pushing subgroup $\Brun_4$.
\end{proof}

\section{An example}
\label{section:appendix}

We end with an illustrative example of a push-map $\Gamma \in K_5$ of the form guaranteed by Proposition~\ref{prop:four-removal}.  We will apply $\Gamma$ to an arc $\beta$ whose corresponding Moody polynomial admits cancellation, and we will see that after applying $\Gamma$ there is no longer any cancellation.  Let $\beta$ denote the arc shown in Figure~\ref{fig:beta-example}.  We recall that an arc $\beta$ does not uniquely determine a braid $\Phi$ such that $\beta = (\beta_*^3) \Phi$; nevertheless in what follows it will be useful to make a choice of such a $\Phi$ to simplify notation.

To simplify the calculation while explaining the important parts of the construction, we have chosen to specify an arc $\beta$ whose disk sequence satisfies the same properties as those guaranteed by Lemma~\ref{lemma:only 4}, but our choice of $\beta$ does not arise as $(\beta_*^3) \Phi$ for any choice of $\Phi \in K_4$.  (Recall that we restricted to $K_4$ in Theorem~\ref{theorem:k4-faithful} in order to ensure that $\Phi$ can be written as a proper product $\Phi = \Phi' \cdot \Gamma_1$, which allows us to apply Lemma~\ref{lemma:only 4}.)  We make this choice for illustrative purposes because the simplest example of an arc $\beta$ admitting cancellations, and arising from a proper product of the form $\Phi = \Phi' \cdot \Gamma_1 \in K_4$, would have many more intersections.  

\begin{figure}[htpb!]
    \centering
    \includegraphics[width=0.7\linewidth]{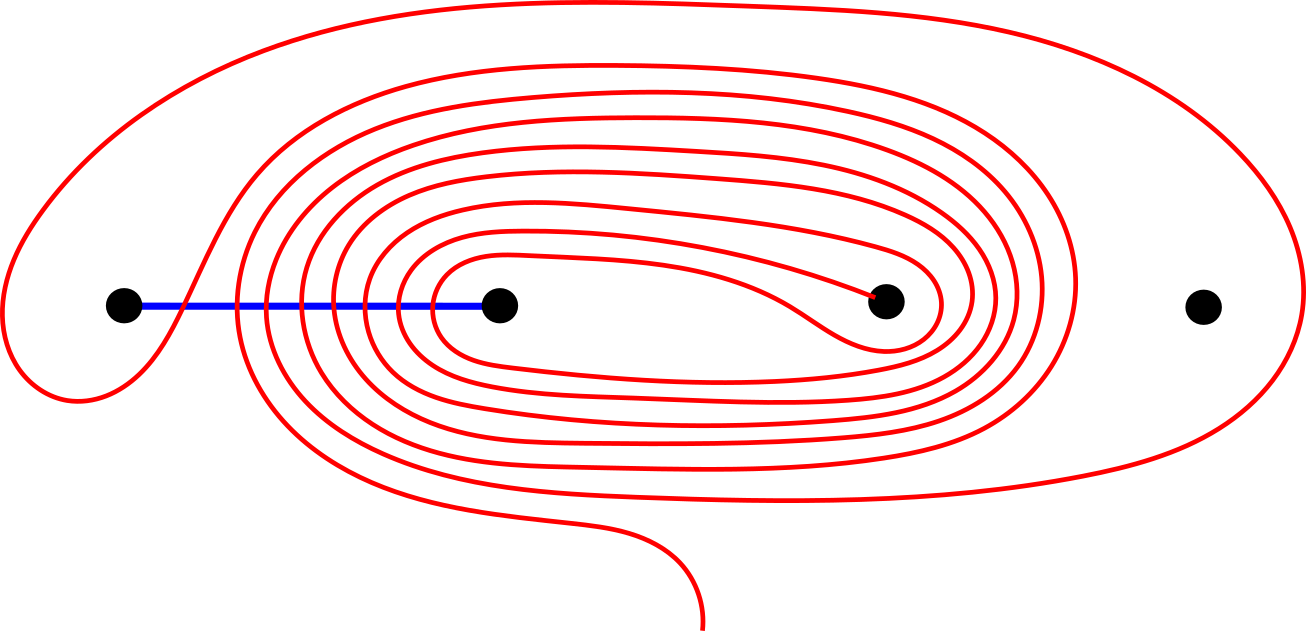}
    \caption{An example of an arc $\beta$ whose Moody polynomial admits a cancellation.}
    \label{fig:beta-example}
\end{figure}

Note that any braid $\Phi$ inducing the arc $\beta$ does not satisfy the parity condition. Here and throughout this appendix, we write the terms of the unsimplified Moody polynomial in order corresponding to points of intersection of $\alpha$ with $\beta$ as we travel along $\beta$ from the basepoint to $p_3$.  We compute the Moody polynomial for $\Phi$ as follows, noting that the linear terms appear with opposite signs, yielding a cancellation.
\begin{eqnarray*}
    \Moody_{\Phi} &=& t^0 + t^2 + t^4 - t^3 - t + t^{-3} + t^{-1} + t \\
    &=& 1 + t^2 + t^4 - t^3 + t^{-3} + t^{-1}.
\end{eqnarray*}

Following the procedure given in the proof of Proposition~\ref{prop:four-removal}, we obtain the loop $\Gamma\in K_5$ shown in Figure~\ref{fig:gamma}. 
\begin{figure}[htpb!]
    \centering
    \begin{overpic}[width=0.7\linewidth]{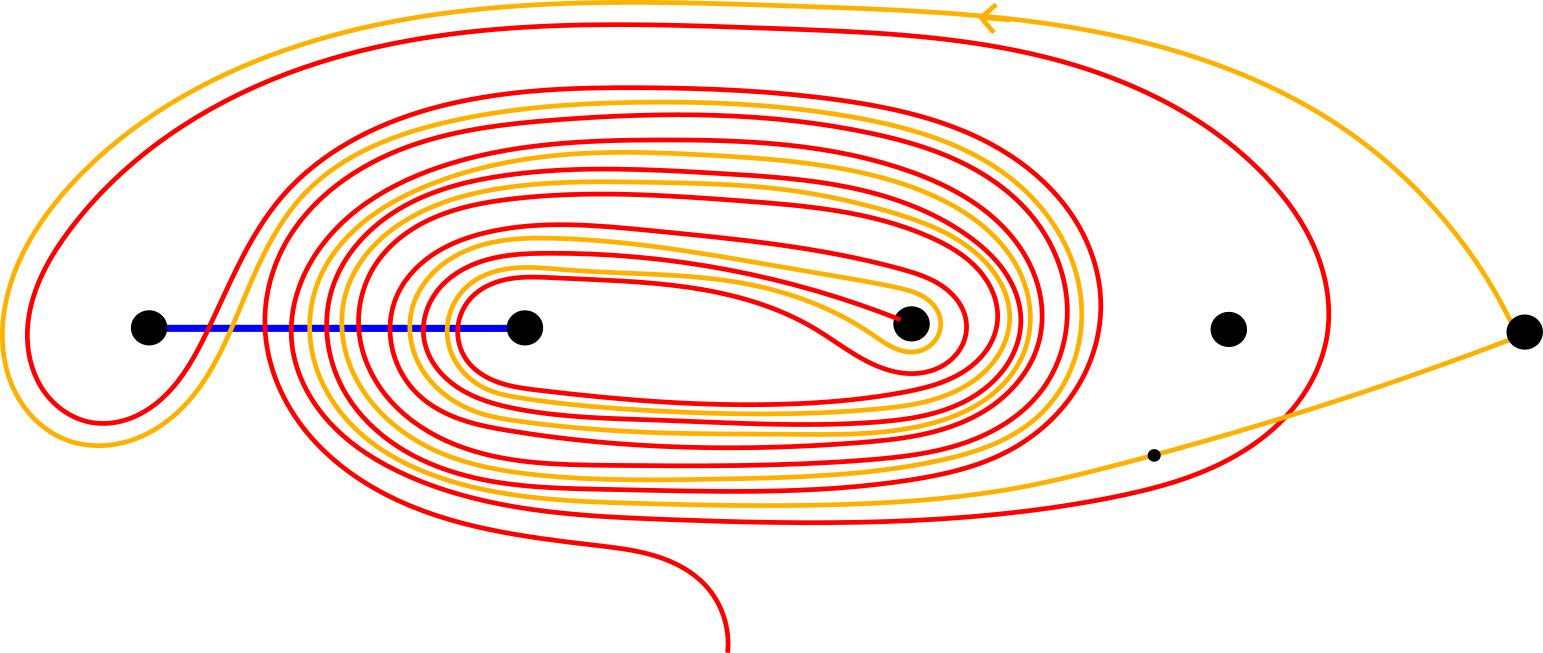}
        \put(74, 15){\small $q$}
    \end{overpic}
    \caption{The push map $\Gamma\in K_5$ shown in gold, with the point $q$ marked from the proof of Proposition \ref{prop:four-removal}.}
    \label{fig:gamma}
\end{figure}
Pushing $\beta$ along $\Gamma$ gives us $\beta \cdot \Gamma$ as shown in Figure~\ref{fig:beta-gamma}, which now satisfies the parity condition.   
\begin{figure}[htpb!]
    \centering
        \includegraphics[width=.8\linewidth]{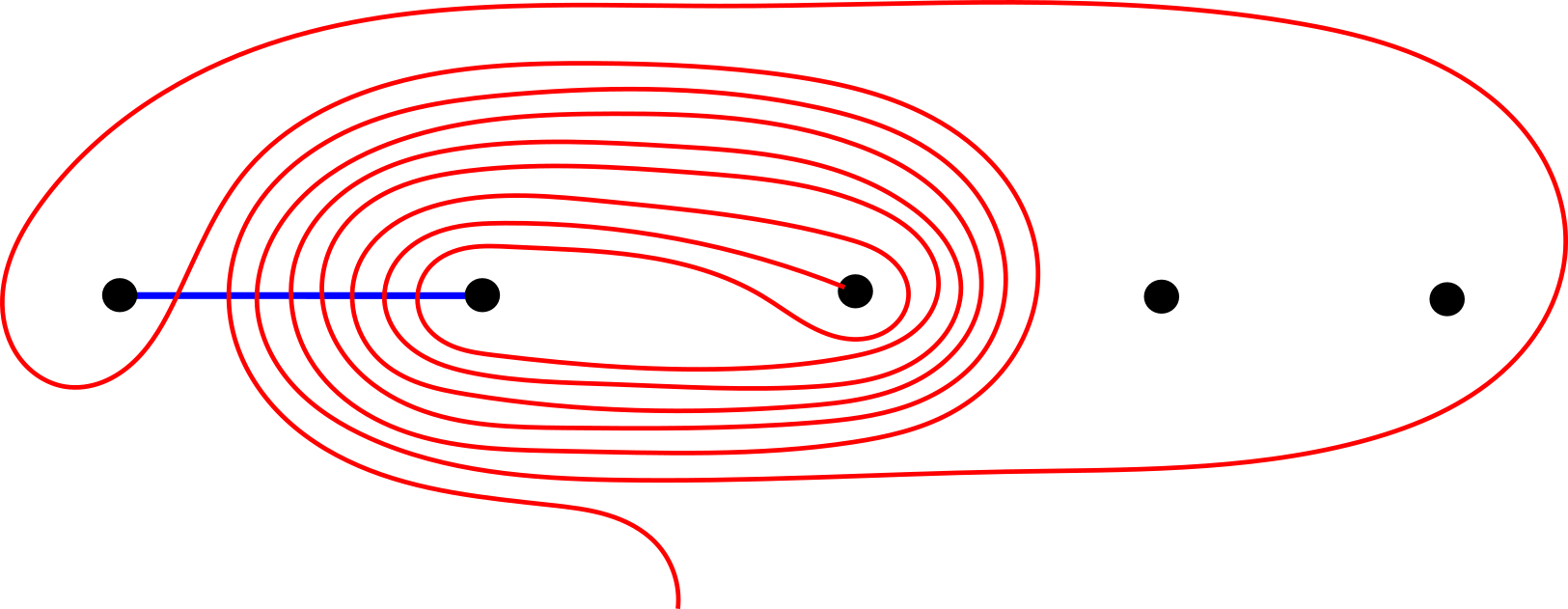}
        \caption{The image $(\beta) \Gamma$ in $D_5$}
    \label{fig:beta-gamma}
\end{figure}
We next compute the Moody polynomial of the arc $\Phi \cdot \Gamma$: 
\begin{eqnarray*}
    \Moody_{\Phi \cdot \Gamma}(t) &=& t^0 + t^2 + t^4 - t^3 - t + t^{-4} + t^{-2} + t^0 \\
    &=& 2 + t^2 + t^4 - t^3 - t + t^{-4} + t^{-2}.
\end{eqnarray*}
Comparing with our calculation of the unsimplified Moody polynomial of $\Phi$ above, we see that the effect of $\Gamma$ on the Moody polynomial has been to increase the exponent of each the last three terms (those following the appearance of the 5-disk) by 1, and there is no longer any cancellation.  

Furthermore, the arc $(\beta_*^3)\cdot \Gamma$ determining the Moody polynomial of $\Gamma$ is shown in Figure~\ref{fig:gamma-image} and satisfies the parity condition as well.
The unsimplified Moody polynomial for the arc $(\beta_*^3)\cdot \Gamma$ has 20 terms; the key point here is that any $\beta$-arc satisfying the parity condition may have any number of terms appearing with equal exponents, but every term with the same exponent will have the same sign.  
\begin{eqnarray*}
\Moody_\Gamma(t) &=& -t^0 - t^2 + t + t^{-1} - t^0 - t^2 + t^3 + t + t^{-1} - t^{-2} - t^0 + t^{-1} + t^{-3} - t^{-2} - t^0 + t + t^{-1} + t^{-3} - t^{-2} - t^0 \\
&=& 2 t^{-3} - 3 t^{-2} + 4t^{-1} - 5 + 3 t - 2t^2 + t^3.
\end{eqnarray*}
\begin{figure}[htpb!]
    \centering
        \includegraphics[width=.8\linewidth]{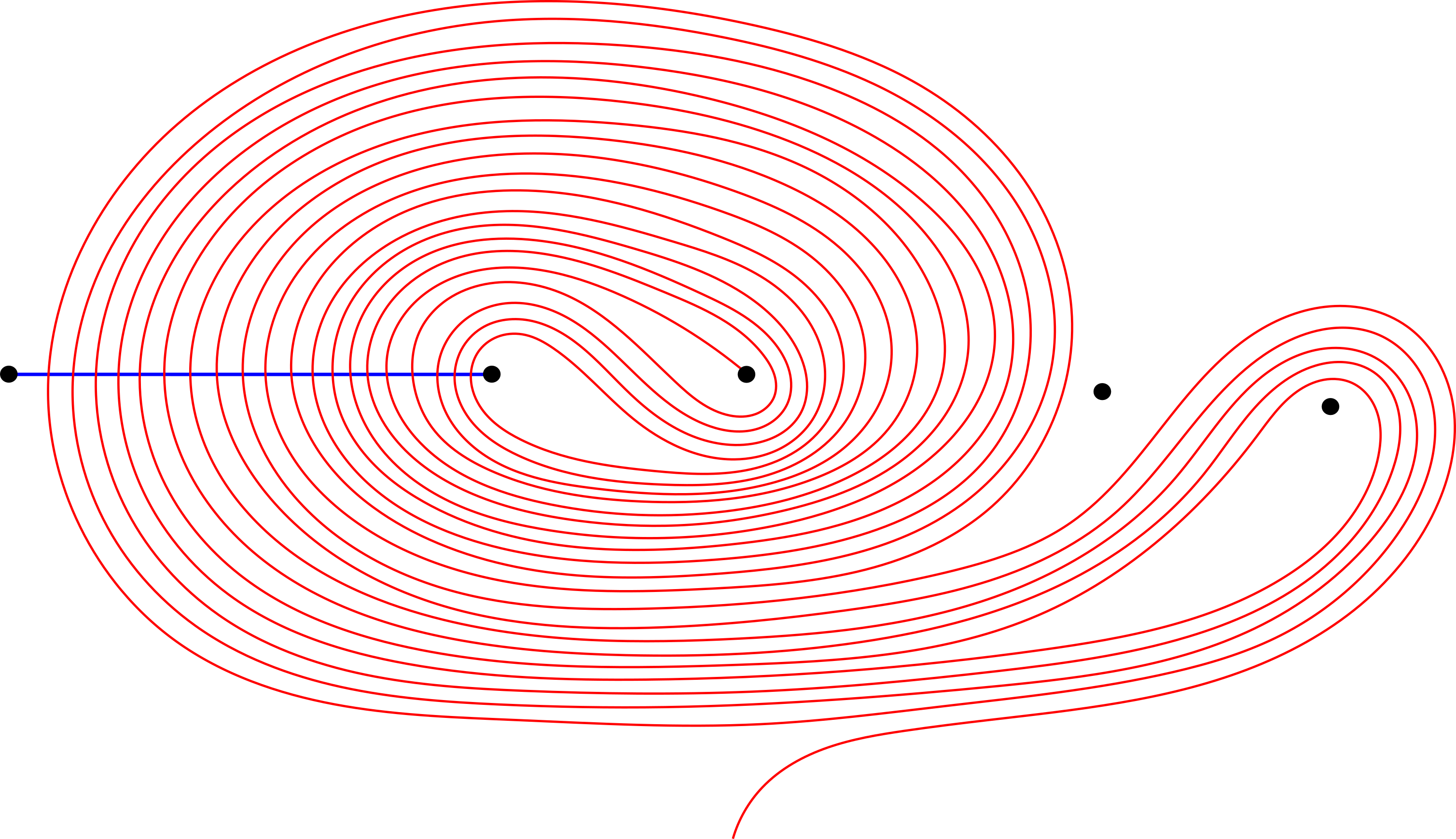}
        \caption{The image $(\beta_*^3)\cdot \Gamma$ in $D_5$.  }
    \label{fig:gamma-image}
\end{figure}

\bibliographystyle{alpha}
\bibliography{Burau4}

\end{document}